\documentclass[11pt]{article}%
\usepackage{amssymb,amsmath,amsfonts,amsthm,array,bm,bbm,color}%
\usepackage{graphicx}
\usepackage{hyperref}
\hypersetup{colorlinks=true, linkcolor=red, anchorcolor=blue, citecolor=blue}
\usepackage{graphicx}
\providecommand{\U}[1]{\protect\rule{.1in}{.1in}}
\numberwithin{equation}{section}

\newtheorem{theorem}{Theorem}[section]
\newtheorem{lemma}[theorem]{Lemma}

\newtheorem{proposition}[theorem]{Proposition}
\newtheorem{remark}[theorem]{Remark}
\newtheorem{example}[theorem]{Example}

\newtheorem{hypothesis}[theorem]{Hypothesis}

\def\<{\langle}
\def\>{\rangle}
\def\d{{\rm d}}
\def\L{\mathcal{L}}

\def\E{\mathbb{E}}
\def\N{\mathbb{N}}
\def\P{\mathbb{P}}
\def\R{\mathbb{R}}

\def\F{\mathcal{F}}
\def\X{X_{t}^{N,i}}
\def\L{L^{1}\cap L^{p}}
\def\eps{\epsilon}
\def\supp{\mathrm{supp}}

\begin{document}

%\title{Scaling Limit of Moderately Interacting Particle Systems with Singular Interaction and Environmental Noise}
%\title{A sample article title with some additional note\thanksref{T1}}

%\thankstext{T1}{A sample of additional note to the title.}

\title{Scaling Limit of Moderately Interacting Particle Systems with Singular Interaction and Environmental Noise
	%A scaling limit for stochastic 2D Boussinesq equations with thermal diffusivity and transport noise
}

\author{Shuchen Guo\footnote{Email: guoshuchen15@mails.ucas.ac.cn. School of Mathematical Sciences, University of the Chinese Academy of Sciences, Beijing, China and Academy of Mathematics and Systems Science, Chinese Academy of Sciences, Beijing, China.} \and Dejun Luo \footnote{Email: luodj@amss.ac.cn. Key Laboratory of RCSDS, Academy of Mathematics and Systems Science, Chinese Academy of Sciences, Beijing, China and School of Mathematical Sciences, University of the Chinese Academy of Sciences, Beijing, China.}}
%\author{Dejun Luo\footnote{Email: luodj@amss.ac.cn. Key Laboratory of RCSDS, Academy of Mathematics and Systems Science, Chinese Academy of Sciences, Beijing, China and School of Mathematical Sciences, University of the Chinese Academy of Sciences, Beijing, China.} }

\maketitle

\vspace{-20pt}

\begin{abstract}
	We consider  moderately interacting particle systems with singular interaction kernel and environmental noise. It is shown that the mollified empirical measures converge in strong norms to the unique (local) solutions of nonlinear Fokker-Planck equations. The approach works for the Biot-Savart and repulsive Poisson kernels.
\end{abstract}


%%%%%%%%%%%%%%%%%%%%%%%%%%%%%%%%%%%%%%%%%%%%%%
%% Please use \tableofcontents for articles %%
%% with 50 pages and more                   %%
%%%%%%%%%%%%%%%%%%%%%%%%%%%%%%%%%%%%%%%%%%%%%%
\tableofcontents

%%%%%%%%%%%%%%%%%%%%%%%%%%%%%%%%%%%%%%%%%%%%%%
%%%% Main text entry area:

\section{Introduction}

In this paper we are concerned with stochastic particle systems on $\R^d\, (d\geq 2)$ with moderate interaction (in the sense of \cite{Oel85, MRC87}) and the so-called \emph{environmental noise}: for any $N\in \N$,
\begin{equation}\label{IPS-0}
	\d X_t^{N,i} = \frac{1}{N}\sum_{j=1}^{N} \big(K\ast  V^N \big) \big( X_t^{N,i}- X_t^{N,j} \big)\, \d t + \d W\big(t, X_t^{N,i} \big), \quad 1\leq i\leq N,
\end{equation}
where $K$ is a (suitably smoothed) interaction kernel, $V^N$ is some convolution kernel to be specified later and $W(t,x)$ is a space-time noise, white in time and colored in space, modeling some random environment in which the particles are embedded. As a comparison, let us recall that the classical mean field stochastic particle system reads as
\begin{equation}\label{IPS-1}
	\d X_t^{N,i} = \frac{1}{N}\sum_{j=1, j\neq i}^{N} K\big( X_t^{N,i}- X_t^{N,j} \big)\,\d t + \d W^i_t,\quad 1\leq i\leq N,
\end{equation}
where $\{W^i_\cdot\}_{i\geq 1}$ is a family of mutually independent standard Brownian motions on $\R^d$. The motivation for considering particle systems with environmental noise comes from Lagrangian point vortex formulation of stochastic Euler equations, see \cite[p.1450]{FGP11} for a heuristic argument. It was stated in the fourth paragraph of \cite[p.1425]{FHM14} that noises on close point vortices should be quite correlated, a fact that might be modeled by environmental noises with suitable spatial covariance. We mention that there are also lots of studies on mean field games with common noise, i.e. there is a common (multiplicative) Brownian noise (independent of $\{W^i\}_{i\geq 1}$) in each equation of \eqref{IPS-1}, see \cite{CDL16, CD18} for systematic treatments. Particle systems with environmental noise (and mean field interaction) have been studied in some recent works under various conditions (see e.g. \cite{CF16, CM20, RM20}); in these papers, however, the noise part was not rescaled with respect to the number of particles and thus the limit equations of empirical measures are stochastic partial differential equations (SPDEs), unlike in the classical case \eqref{IPS-1} where the limit equations are deterministic. Recently, it was shown that some SPDEs with multiplicative noise of transport type, under a suitable scaling of the noise, converge to limit equations with an extra viscous part, see  \cite{FL20} for the vorticity form of stochastic 2D Euler equations and \cite{G20} for stochastic linear transport equations. Such ideas have been adapted in \cite{FL21} to the stochastic point vortex system with suitably rescaled environmental noise, showing that the weak limits of empirical measures satisfy the deterministic 2D Navier-Stokes equation; cf. \cite{FL19, FL21a} for scaling limits different from the mean field regime. Our purpose is to seek for a certain rescaling of the noise in \eqref{IPS-0} and prove, for some singular interaction kernels (including the Biot-Savart and repulsive Poisson kernels), strong convergence of mollified empirical measures to solutions of deterministic nonlinear Fokker-Planck equations.

Before stating our hypotheses and main results, we briefly recall some facts on interacting particle systems \eqref{IPS-1} which are widely used to model various phenomena in natural and social sciences. As the number of particles in a system is usually very large, one tries to derive macroscopic equations by studying the limit behavior of empirical measures, often described by some nonlinear PDEs. In the case of Lipschitz interaction kernel, the coupling method was popularized in \cite{Sznitman91} to show the mean field limit and propagation of chaos, namely, particles become independent of each other in the limit. A well known example of singular interaction kernel is the Biot-Savart kernel for point vortices, in which case the limit equation is the vorticity form of 2D Navier-Stokes equation, see e.g. \cite{Osada, Mel00, FHM14}. For particle systems related to other singular kernels, one can refer to \cite{Sznitman86, BT97} for Burgers equation, to \cite{CP16, FJ17, BJW19} for the Keller-Segel model on chemotaxis. In the case of singular interactions, it is often difficult to get quantitative convergence rates of empirical measures, but there are some remarkable progresses in recent years, see for instances \cite{Ser17, JW18, Ser20}. The large deviation principle was proved in \cite{LW20} for the empirical measures, and the long time behavior of the kinetic Fokker-Planck equation with mean field interaction was studied in \cite{GLWZ21}. We refer to the survey \cite{JW17} for more literatures on mean field limits.

Since its introduction by Oelschl\"ager \cite{Oel85}, particle systems with moderate interaction
$$\d X_t^{N,i} = \frac{1}{N}\sum_{j=1}^{N} \big(K\ast  V^N \big)\big( X_t^{N,i}- X_t^{N,j} \big)\,\d t + \d W^i_t,\quad 1\leq i\leq N$$
have also attracted lots of attention, see e.g. \cite{MRC87, Oel89, Jourdain98, JM98}. In the recent paper \cite{FLO19}, Flandoli et al. developed a semigroup approach which enables them to show uniform convergence of mollified empirical measures; see \cite{FlLeo19, FOS20, ORT20a, ORT20b} for further applications of this method. In particular, Flandoli et al. \cite{FOS20} considered the moderately interacting point vortex system with a cut-off:
\begin{equation}\label{IPS-2}
	\d X_t^{N,i} = F_A\bigg(\frac{1}{N}\sum_{j=1}^{N} \big(K\ast  V^N \big)\big( X_t^{N,i}- X_t^{N,j} \big) \bigg)\,\d t + \d W^i_t,\quad 1\leq i\leq N,
\end{equation}
where $K$ is now the Biot-Savart kernel and $F_A:\R^2\to \R^2,\, x\mapsto (f_A(x_1), f_A(x_2))$ is a smooth cut-off with threshold $A>0$, and they proved that the smoothed empirical measures converge uniformly to the unique solution of 2D Navier-Stokes equation in vorticity form. Thanks to explicit estimate on solutions to the limit equation, the cut-off does not appear in the limit by choosing a sufficiently large $A>0$. This result was greatly improved in \cite{ORT20b} for singularly interacting particle systems with cut-off, providing quantitative convergence rates.

In the present work, we consider moderately interacting particle systems with singular kernel $K$ and environmental noise $W_N$:
\begin{equation}\label{IPS-3}
	\d X_t^{N,i} = \frac{1}{N}\sum_{j=1}^{N} \big(K_\epsilon \ast  V^N \big) \big( X_t^{N,i}- X_t^{N,j} \big)\, \d t + \d W_N\big(t, X_t^{N,i} \big), \quad 1\leq i\leq N,
\end{equation}
where $K_\epsilon= K\ast  \rho_\epsilon$ for some smoothing kernel $\rho_\epsilon$, and we have written $W_N$ to indicate that the noise will be rescaled along with the number of particles. There are two main differences from \eqref{IPS-2}: the noise in the above system is of environmental type, and the cut-off $F_A$ in the interaction part is replaced by convolution with $\rho_\epsilon$ (see Remark \ref{rem-conv-kernels} below for some comments on the latter). As in \cite{FL21}, for particle systems \eqref{IPS-3} with environmental noise, one difficulty is to show that the noise parts in the equations of empirical measures vanish as $N\to \infty$, a fact that follows easily in the case of independent Brownian noises as in \eqref{IPS-2}. To overcome this difficulty (see Lemma \ref{convratelemma}), we shall make nontrivial use of estimates on Boltzmann entropy obtained in Lemma \ref{upperbound}, where in the proof we need to compute the divergence of the drift vector field corresponding to the interaction part, and it seems difficult to get useful estimates for the cut-off in \eqref{IPS-2}.

\subsection{The precise model}

We first give a more precise form of the noise in the particle system \eqref{IPS-3}. Throughout the paper we assume that the noise $W_N(t,x)$ is white in time, spatially homogeneous and divergence free; define its covariance matrix as
$$Q_N(x-y) = \E[W_N(1,x)\otimes W_N(1,y)],\quad x,y\in \R^d. $$
By general arguments (cf. \cite[Section 1]{LJR02}), there exist (possibly infinitely many) divergence free vector fields $\{\sigma^N_k \}_{k\geq 1}$ on $\R^d$, such that
$$Q_N(x-y)= \sum_k \sigma^N_k(x) \otimes \sigma^N_k(y). $$
We shall assume that $Q_N\in C_b^{2+\delta}(\R^d, \R^{d\times d})$ for some $\delta>0$, then it is not difficult to see that $\sum_k \|\sigma^N_k \|_{C_b^{1+\delta/2}}^2<+\infty$; cf. \cite[Section 2.1]{CF16} for some related discussions. Taking a sequence of independent standard Brownian motions $\{W^k \}_{k\geq 1}$ on some filtered probability space $(\Omega, \F, (\F_t), \P)$, we can represent $W_N(t,x)$ as
$$W_N(t,x) = \sum_k \sigma^N_k(x) W^k_t. $$
The above series converge in the mean square sense, locally uniformly in $(t,x)$. Therefore, the stochastic particle system we consider has the precise form: for $1\leq i\leq N$,
\begin{equation}\label{IPS}
	\d X_t^{N,i} = \frac{1}{N}\sum_{j=1}^{N} \big(K_\epsilon \ast  V^N \big) \big( X_t^{N,i}- X_t^{N,j} \big)\, \d t + \sum_k \sigma^N_k \big( X_t^{N,i} \big) \,\d W^k_t, \quad X_0^{N,i} = X^i_0,
\end{equation}
where $\{X^i_0\}_{i\geq 1}$ is a family of $\F_0$-measurable i.i.d. random variables on $\R^d$. The equation \eqref{IPS} can also be written in the Stratonovich form which coincides with the It\^o form. Thanks to the convolution, the drift coefficient in \eqref{IPS} is smooth with bounded derivatives; this plus the regularity assumptions on the diffusion coefficients implies that the above system generates a stochastic flow of diffeomorphisms on $\R^d$, see \cite[Section 4.6]{Kunita90} and Lemma \ref{density} below for more information on the flow.

As usual, we denote the empirical measure by
\begin{equation*}%\label{empiricalmeasure}
	S_{t}^{N}=\frac{1}{N} \sum_{i=1}^{N} \delta_{X_{t}^{N, i}},
\end{equation*}
and rewrite \eqref{IPS} in the following more compact form: for $1\leq i\leq N$,
\begin{equation}\label{compactform}
	\d X_{t}^{N,i}= \big(K_{\epsilon}\ast  V^N\ast  S_t^N\big)\big(X_t^{N,i} \big)\, \d t +\sum_k \sigma^N_k \big(X_{t}^{N,i} \big)\, \d W_t^k, \quad X_0^{N,i} = X^i_0.
\end{equation}
We will study the mollified empirical measure $\omega_{t}^{N}$ defined as
\begin{equation}\label{mollifiedmeasure}
	\omega_{t}^{N}(x):=\left(V^{N}\ast  S_{t}^{N}\right)(x)=\int_{\R^d} V^N(x-y)\, \d S_{t}^{N}(y).
\end{equation}
Under suitable hypotheses on the interaction kernel $K$, on the convolution kernels $\rho_\eps, V^N$ and on the covariance function $Q_N$, we will prove that $\omega^N_{\cdot}$ converges, as $N\rightarrow\infty$, in some strong norm to the unique (local) solution of the deterministic nonlinear Fokker-Planck equation:
\begin{equation}\label{PDE}
	\left\{\begin{array}{l}
		\partial_{t} \omega+ \nabla\cdot ((K \ast  \omega) \omega)=\nu \Delta \omega,\\
		\omega(0, \cdot)=\omega_{0}(\cdot),
	\end{array}\right.
\end{equation}
where $\nu>0$ is a constant and $\omega_0:\R^d\to \R$ is some function. The solution of this equation is understood as follows: given $\omega_0\in L^1\cap L^p(\R^d)$, a function $\omega: [0,T]\times \R^d\to \R$ is called a mild solution of \eqref{PDE} if $\omega \in C\big([0,T], L^1\cap L^p(\R^d)\big)$ and the following identity holds:
\begin{equation}\label{omegat}
	\omega_t = e^{\nu t\Delta}\omega_0 - \int_0^t \nabla \cdot e^{\nu (t-s)\Delta} ((K \ast  \omega_s) \omega_s)\,\d s, \quad t\in [0,T].
\end{equation}
Here $\{e^{\nu t\Delta} \}_{t\geq 0}$ is the heat semigroup generated by $\nu \Delta$, see Section \ref{notationsandlemma} below for more details.

\subsection{Hypotheses and main result}

We first state our hypotheses on the interaction kernel $K$. For any $R>0$, denote by $B(R)$ the closed ball in $\R^d$ centered at the origin with radius $R$ and $B(R)^c$ its complement.

\begin{hypothesis}\label{hypothesis0}
	We assume that the kernel $K:\R^d \to \R^d$ is smooth away from the origin, with the following properties:
	\begin{itemize}
		\item[1.] there exists $p>2$ such that $K\in L^{p^{\prime}}(B(1))$, where $p^{\prime}$ is the conjugate number of $p$;
		\item[2.] $K\in L^{\infty}(B(1)^{c})$;
		\item[3.] The divergence of $K$ is nonnegative in the sense of distribution, namely $\nabla\cdot K \geq 0$.
	\end{itemize}
\end{hypothesis}

\begin{example}\label{exa-intera-kernels}
	Here are some examples of kernels satisfying the above conditions.
	\begin{itemize}
		\item The Biot-Savart kernel $K(x)= \frac1{2\pi} \frac{x^\perp}{|x|^2}$ where $x^\perp =(-x_2, x_1)$ for any $x\in \R^2\setminus \{0\}$.
		\item The repulsive Poisson kernel $K(x)= C_d \frac{x}{|x|^d},\, x\in \R^d\setminus \{0\}$, where $C_d>0$. Note that the first condition holds with $p>d$ and $\nabla\cdot K(x) = C_d^{\prime}\delta_0$, where $C_d^{\prime}>0$ and $\delta_0$ is the Dirac delta mass at the origin.
		\item More generally, let $K(x) = \nabla V_s(x)$, where
		$$V_s(x) = \begin{cases}
			-|x|^{-s}, & s\in (0,d), \\
			\log |x|, & s=0,
		\end{cases} \quad x\in \R^d\setminus \{0\} $$
		is the Riesz potential. In order that the conditions are fulfilled, we need  $s\in [0,d-2]$ and $p>s+2$. When $s=d-2$, it reduces to the repulsive Poisson kernel; when $0<s<d-2$, $\nabla\cdot K(x)=C_{s,d}\frac{1}{|x|^{s+2}}$ for some constant $C_{s,d}>0$.
	\end{itemize}
\end{example}

Next we give the conditions on the convolution kernels and the covariance functions $Q_N$. We write $I_d$ for the $d$-dimensional identity matrix. Throughout the paper, we fix a parameter $m>2$.

\begin{hypothesis}\label{hypothesis}
	\begin{itemize}
		\item[1.] We define the regularized kernel $K_{\epsilon}:=K\ast  \rho_{\epsilon}$ with mollifier  $\rho_{\epsilon}(x)= \epsilon^{-d}\rho(\epsilon^{-1} x)$, where $\epsilon>0$ and $\rho\in C_c^{\infty}(\R^d,\R_{+})$ is a probability density.
		\item[2.] The interaction potential $V^N(x)=N^{d\beta}V(N^{\beta}x)$, where $V\in C_c^{\infty}(\R^d,\R_{+})$ is a probability density and $\beta$ is taken as \begin{equation}\label{hypobetarange}
			0<\beta\leq \frac{1}{4m(d+2)}.
		\end{equation}
		\item[3.] The covariance functions $Q_N\in C_b^{2+\delta}(\R^d, \R^{d\times d})$ satisfy that $Q_N(0) = 2\nu I_d$ for some $\nu>0$, and $Q_N$ corresponds to divergence free vector fields; moreover, for any $r\geq 2$, there exists $C_r>0$ such that $\|Q_N \|^r_{L^r} \leq C_r e^{-2N}$ for any $N\geq 1$.
	\end{itemize}
\end{hypothesis}

\begin{remark}\label{rem-conv-kernels}
	Here are some comments on the above hypotheses.
	\begin{itemize}
		\item We need to regularize the singular interaction kernel $K$ by convolution since, in the proof of the main result via the semigroup method, quantities like $\|K_\epsilon \|_\infty$ and $\|\nabla K_\epsilon \|_\infty$ will appear, see e.g. \eqref{kg} and \eqref{kk}. That is also the reason why a cut-off as in \eqref{IPS-2} has been introduced in \cite{FOS20, ORT20b}.
		\item The presence of potential $V^N$ is the main feature of moderate interaction; actually, the range of $\beta$ can be slightly larger than \eqref{hypobetarange}; see  Remark \ref{betarange} for more details.
		\item The last condition on the covariance matrix function may look strange at first sight; it will be used in the proof of Lemma \ref{convratelemma} where we need rapid decrease of $\|Q_N \|_{L^r}$ to cancel certain quantities. Heuristically, we require that the random external force $W_N(t,x)$ has constant trace and exponentially vanishing covariance. We shall give in Section \ref{sectionnoise} an example of such covariance functions in terms of the Kraichnan noise.
	\end{itemize}
\end{remark}

Finally we present the assumptions concerning the initial particles $X^{i}_0,\, i\geq 1$. Recall the definition of mollified empirical measure in \eqref{mollifiedmeasure}.

\begin{hypothesis}\label{hypothesis-initial}
	\begin{itemize}
		\item[1.] Let $\{X^i_0 \}_{i\geq 1}$ be a family of i.i.d. random variables on $\R^d$ with finite moment of order $m(d+1)$ and common density function $\omega_0 \in L^1\cap L^p(\R^d,\R_+)$, where $p>2$ is the same as in Hypothese \ref{hypothesis0}.
		\item[2.] Assume $\omega^N_0$ converges to $\omega_0$ with rate $\zeta_N$, i.e.% There exist $\lambda >0$ and $C>0$ such that
		$$\zeta_N:=\big\|\omega^N_0 - \omega_0 \big\|_{L^m(\Omega, L^1\cap L^p)} \rightarrow 0  \quad\text{as}\quad N \rightarrow\infty. $$
	\end{itemize}
\end{hypothesis}

\begin{remark}\label{rem-initial}
	\begin{itemize}
		\item The first condition in Hypotheses \ref{hypothesis-initial} implies, together with the range of $\beta$ in Hypotheses \ref{hypothesis}, that $\{\omega^N_0 \}_{N\geq 1}$ is uniformly bounded in $L^m(\Omega, \L)$. Indeed, by Sobolev embedding  $H^{d(\frac{1}{2}-\frac{1}{p}),2}(\R^d)\subset L^p(\R^d)$, it suffices to show that $\{\omega^N_0 \}_{N\geq 1}$ is uniformly bounded in $L^m\big(\Omega,  H^{d(\frac{1}{2}-\frac{1}{p}),2}\big)$; a proof of the latter can be found in \cite[Lemma 6]{FlLeo19} in the special case $\omega_{0}\in C^{\infty}_c(\R^d)$.  We shall give in Appendix \ref{boundomega0} a sketched proof for $\omega_0\in L^1\cap L^p$ and $m\leq p =3$.
		\item If $\omega_0 \in C_c^1(\R^d)$, then one can prove that $\zeta_N\leq CN^{-\lambda}$ for some $C>0$ and $\lambda>0$. We show such an estimate in the special case $m=p=4$, see Appendix \ref{boundomega0omega}.
	\end{itemize}
\end{remark}

Before stating our main result, we remark that, for general interaction kernel as in Hypotheses \ref{hypothesis0}, the limit equation \eqref{PDE} may not have a global solution; however, thanks to the first two conditions in Hypotheses \ref{hypothesis0}, local existence is guaranteed by \cite[Proposition 1.2]{ORT20b}. We shall denote by $T_{max}$ the maximal time of existence of mild solution to \eqref{PDE}.

\begin{theorem}\label{thm-main}
	Assume Hypotheses \ref{hypothesis0}, \ref{hypothesis} and \ref{hypothesis-initial}, and take
	$$\epsilon = \epsilon(N) \sim \big[(\log N)\wedge (-\log \zeta_N)\big]^{-\frac{p^{\prime}}{2d}}.$$
	Then for any $T\in (0,T_{max})$, we have
	$$\lim_{N\to \infty} \bigg\| \sup_{t\in [0,T]} \big\| \omega^N_t - \omega_t \big\|_{L^1\cap L^p} \bigg\|_{L^m(\Omega)} =0. $$
\end{theorem}

Unfortunately, we are unable to get a polynomial convergence rate as in \cite[Theorem 1.3]{ORT20b}, even if we assume such a convergence rate for the initial data as in the second item of Remark \ref{rem-initial}. The reason is that we shall actually prove the result in two steps: firstly, prove that the mollified empirical measures $\omega^N$, for $\epsilon= \epsilon(N)$ chosen above, are close to the solution of the nonlinear Fokker-Planck equation with smooth kernel:
\begin{equation}\label{PDE-eps}
	\left\{\begin{array}{l}
		\partial_t \omega^\epsilon+ \nabla\cdot ((K_\epsilon \ast  \omega^\epsilon) \omega^\epsilon)=\nu \Delta \omega^\epsilon,\\
		\omega^\epsilon(0, \cdot)=\omega_0(\cdot);
	\end{array}\right.
\end{equation}
secondly, prove that the solution $\omega^\epsilon$ tends to that of \eqref{PDE} as $\epsilon \to 0$. The convergence rate of the second step is very slow due to the choice of $\epsilon(N)$.

We close this introduction with the structure of the paper. In Section \ref{sectionpreparations} we make some necessary preparations for the proof of Theorem \ref{thm-main}; in particular, the estimate on Boltzmann entropy in Section \ref{sectionentropy} will play an important role in the sequel. Section \ref{sectioncovariance} contains a key estimate on the covariance of random forces on two different particles. Then, we will prove Theorem \ref{thm-main} in Section \ref{sec-proof-main-thm} by following some ideas of the proof of \cite[Theorem 1.3]{ORT20b}. Finally, we provide in the appendices the estimates on stochastic convolutions which are more complex than those in \cite[Section A.2]{ORT20b}, and some estimates on $\omega^N_0$.

\section{Preparations}\label{sectionpreparations}

We introduce some notations and a generalized Gr\"{o}nwall lemma in Section \ref{notationsandlemma}, then we give in Section \ref{sectionnoise} an example of covariance functions of the noises in \eqref{IPS}; Section \ref{sectionkernel} contains some properties of the regularized interaction kernel; finally, the Boltzmann entropy is introduced in Section \ref{sectionentropy}, together with some useful bounds.

\subsection{Functional setting and a Gr\"{o}nwall lemma}\label{notationsandlemma}

We define the operator:
$$A: \mathcal{D}(A) \subset L^{p}(\R^d) \rightarrow L^{p}(\R^d)$$
as $Af:=\nu \Delta f$, which is the infinitesimal generator of the heat semigroup $\{e^{tA}\}_{t\geq 0}$: for any $f\in L^{p}(\R^d)$,
\begin{equation}\label{semigroup}
	\begin{aligned}
		\left(e^{t A} f\right)(x)=&\,\int_{\R^d}g_t(x-y)f(y)\d y,
	\end{aligned}
\end{equation}
where
\begin{equation*}
	\begin{aligned}
		g_t(x-y):=\frac{1}{(4 \nu \pi t)^{\frac{d}{2}}} e^{-\frac{|x-y|^{2} }{4 \nu t}}.
	\end{aligned}
\end{equation*}
For any $r\in\R$, the Bessel potential operator is given by $(\mathrm{I}-A)^{r/2}$ where $\mathrm{I}$ is the identity operator. The domain of this operator $\mathcal{D}\left((\mathrm{I}-A)^{r / 2}\right)$ is called the Bessel potential space
$$
H^{r, p}\big(\R^{d}\big):=\left\{f \in \mathcal{S}^{\prime}\big(\R^{d}\big): \mathcal{F}^{-1}\left[\left(1+|\cdot|^{2}\right)^{\frac{r}{2}} \mathcal{F} f\right] \in L^{p}\big(\R^{d}\big)\right\},
$$
where $\mathcal{S}^{\prime} \big(\R^{d}\big)$ denotes the space of tempered distributions and $\F f$ is the Fourier transform of $f$. The norm in this space is
$$
\|f \|_{r, p}=\left\|\mathcal{F}^{-1}\left[\left(1+|\cdot|^{2}\right)^{\frac{r }{2}} \mathcal{F} f (\cdot)\right]\right\|_{L^{p}\left(\R^{d}\right)}.
$$
For any $t\in[0,T]$, $\varepsilon>0$ and $q\geq1$, we have the following estimate:
\begin{equation}\label{lpestimate}
	\left\|(\mathrm{I}-A)^{\varepsilon} e^{t A}\right\|_{L^{q} \rightarrow L^{q}} \leq \frac{C_{\varepsilon,  \nu,T, q}}{t^{\varepsilon}},
\end{equation}
where the details can be found in \cite{PA12}.

Here we quote a generalized Gr\"{o}nwall lemma (see similar statements in  \cite[Lemma 7.1.1]{DH06} and \cite[Corollary 2]{HY07}):

\begin{lemma}\label{Gronwall}
	Assume $b\geq0$, $\alpha>0$ and $a(t)$ is a nondecreasing locally integrable function on $\left[0,T\right)$, and assume $u(t)$ is nonnegative and  locally integrable on $\left[0,T\right)$ with
	$$
	u(t)\leq a(t)+b\int_{0}^{t}(t-s)^{\alpha-1}u(s)\d s
	$$
	on this interval. Then
	\begin{equation}\label{mittag}
		\begin{aligned}
			u(t)\leq a(t)\, E_{\alpha} (b\Gamma(\alpha)t^{\alpha} ),\\
		\end{aligned}
	\end{equation}
	where $E_{\alpha}$ is the Mittag-Leffler function defined by $E_{\alpha}(z)=\sum_{k=0}^{\infty}\frac{z^k}{\Gamma(k\alpha+1)}$.
\end{lemma}

For $\alpha=\frac{1}{2}$, \cite{HJ11} gives the identity
\begin{equation*}
	E_{\frac{1}{2}}(z)=e^{z^2}\left(1+\frac{2}{\sqrt{\pi}}\int_{0}^{z}e^{-t^2}\d t\right),
\end{equation*}
then we have the simple estimate
\begin{equation}\label{estimationmittag}
	E_{\frac{1}{2}}\left(z\right)\leq 2e^{z^2},
\end{equation}
and \eqref{mittag} becomes
\begin{equation}\label{mittagbound}
	\begin{aligned}
		u(t)\leq 2a(t)e^{\pi b^2t}.\\
	\end{aligned}
\end{equation}

\subsection{Kraichnan type noises}\label{sectionnoise}

We give in this part an example of covariance functions satisfying the last condition in Hypothesis \ref{hypothesis}. Let us consider the following covariance function on $\R^d$: for some constant $\alpha >2$,
\begin{equation}\label{cov}
	\begin{aligned}
		Q_N(z):=&\,e^{\alpha N}\int_{e^{N}\leq|k|<\infty} \frac{1}{|k|^{d+\alpha}}e^{ik\cdot z}\left(I_d-\frac{k\otimes k}{|k|^2}\right)\d k
		\\=&\,\int_{1\leq|k|<\infty} \frac{1}{|k|^{d+\alpha}}e^{ie^Nk\cdot z}\left(I_d-\frac{k\otimes k}{|k|^2}\right)\d k,
	\end{aligned}
\end{equation}
where $I_d$ denotes $d\times d$ identity matrix. By symmetry of covariance matrix, we notice that
$$Q_N(0)=\int_{1\leq|k|<\infty}\frac{1}{|k|^{d+\alpha}}\left(I_d-\frac{k\otimes k}{|k|^2}\right)\d k=2\nu I_d$$
for a suitable  constant $\nu$ independent of $N$. Then there exists a family of divergence free vector fields $\{\sigma_{k}^{N}\}_{k\geq1}$, such that
\begin{equation}\label{defQN}
	Q_N(x-y)=\sum_{k}\sigma_{k}^{N}(x)\otimes\sigma_{k}^{N}(y).
\end{equation}
It is easy to see that
$$
\lim\limits_{N\rightarrow\infty}\left|Q_N(x-y)\right|=0 \quad \mbox{for all } x\neq y.
$$
%Here is an important remark which will be used in the next section:

\begin{remark}\label{intofcov}
	Define a matrix-valued function
	$$f(k):={\bf 1}_{\{1\leq|k|<\infty\}}\frac{1}{|k|^{d+\alpha}}\left(I_d-\frac{k\otimes k}{|k|^2}\right),$$
	and $f_{ij}$  its  entries  $i,j =1,\ldots,d$. Note that $f_{ij}\in L^{r}(\R^d)$ for any $r\in \left[1,2\right]$. It is well known in Fourier analysis that $ \|\F f \|_{L^{r^{\prime}}}\leq  \|f \|_{L^{r}}$, where $r^{\prime}\geq 2$ is the conjugate number of $r$. We can regard the covariance function as a Fourier transform: $ (Q_N(z) )_{ij}= (\F f_{ij} )(-e^Nz)$; then there exists a constant $C_r$ independent of $N$ such that $ \|Q_N \|^{r^{\prime}}_{L^{r^{\prime}}}= e^{-dN} \|\F f \|^{r^{\prime}}_{L^{r^{\prime}}} 	\leq C_re^{-dN}$.
\end{remark}

\subsection{The regularized interaction kernel}\label{sectionkernel}

Recalling Hypotheses \ref{hypothesis0} and \ref{hypothesis}, we have already defined the regularized kernel $K_{\epsilon}:=K\ast \rho_{\epsilon}$. H\"older's inequality yields
\begin{equation*}
	\begin{aligned}
		\left|K_{\epsilon}(x)\right|&\,\leq\int_{B(1)}|K(y)|\rho_{\epsilon}(x-y)\d y+\int_{B(1)^{c}}|K(y)|\rho_{\epsilon}(x-y)\d y\\
		&\,\leq\left\|\rho_{\epsilon}\right\|_{L^p} \|K \|_{L^{p^{\prime}}(B(1))} +\left\|\rho_{\epsilon} \right\|_{L^1} \sup_{y \in B(1)^{c}}|K(y)|\\
		&\,\leq C_{K}\left(\left\|\rho_{\epsilon}\right\|_{L^p} +1\right),
	\end{aligned}
\end{equation*}
and hence for any small $\epsilon>0$, $\left\|K_{\epsilon}\right\|_{\infty}\leq C\epsilon^{-d/p^{\prime}}$, where the constant $C$ is independent of $\epsilon$. Similarly,
\begin{equation*}
	\begin{aligned}
		\left|\nabla K_{\epsilon}(x)\right|&\,\leq\int_{B(1)}|K(y)|\, |\nabla \rho_{\epsilon}(x-y)|\d y+\int_{B(1)^{c}}|K(y)|\, |\nabla \rho_{\epsilon}(x-y)|\d y\\
		&\,\leq C_{K}\left(\left\|\nabla \rho_{\epsilon}\right\|_{L^{p}}+\left\|\nabla \rho_{\epsilon}\right\|_{L^1}\right),
	\end{aligned}
\end{equation*}
which implies that, for any small enough $\epsilon>0$, $\left\|\nabla K_{\epsilon}\right\|_{\infty}\leq C\epsilon^{-1-d/p^{\prime}}$.

Next, for any $f\in \L$, we can prove in the same way that
\begin{equation}\label{l1p}
	\left\|K_{\epsilon}\ast f\right\|_{\infty}=	\left\|K\ast \left(\rho_{\epsilon}\ast f\right)\right\|_{\infty}\leq C_K\left\|\rho_{\eps}\ast f\right\|_{\L}\leq C_K\left\|f\right\|_{\L}.
\end{equation}

\begin{lemma}\label{1stmoment}
	Assume Hypotheses \ref{hypothesis} and \ref{hypothesis-initial};	for any $t\in[0,T]$,
	$$\E\bigg[\sup_{t \in[0, T]}	\big| X_t^{N,i} \big|^{m(d+1)}\bigg] \leq\,C_{m,\nu,d,T}\left(1+\left\|K_{\epsilon}\right\|_{\infty}^{m(d+1)}  \right).$$
\end{lemma}

\begin{proof}
	By \eqref{IPS},
	\begin{equation}\label{absolutevalue}
		\begin{aligned}
			\big| X_t^{N,i}\big|\leq\,\left| X_{0}^{i}\right|+\int_0^t \frac{1}{N}\sum_{j=1}^{N}\left|\left(K_{\epsilon}\ast V^N\right)( X_s^{N,i}- X_s^{N,j})\right|\d s+\bigg|\int_0^t \sum_{k}\sigma_{k}^{N}(X_s^{N,i})\d W_s^k\bigg|.
		\end{aligned}
	\end{equation}
	By Hypotheses \ref{hypothesis-initial} we have $\E\left| X_{0}^{i}\right|^{m(d+1)}\leq C$; moreover,
	the second term on the right-hand side of \eqref{absolutevalue} can be estimated as
	$$
	\E\Bigg[\bigg(\sup_{t \in[0, T]} \int_0^t\frac{1}{N} \sum_{j=1}^{N} \left|\left(K_{\epsilon}\ast V^N\right)( X_s^{N,i}- X_s^{N,j})\right| \d s \bigg)^{m(d+1)} \Bigg] \leq T^{m(d+1)}\left\|K_{\epsilon}\right\|^{m(d+1)}_{\infty}.
	$$
	Finally, for the third term, by BDG inequality, we have
	$$
	\begin{aligned}
		\E\Bigg[\sup_{t \in[0, T]} \bigg|\int_0^t \sum_{k}\sigma_{k}^{N}(X_s^{N,i})\d W_s^k \bigg|^{m(d+1)} \Bigg] \leq&\, C_{m,d}\E\Bigg[ \bigg(\int_0^T \sum_{k} \left| \sigma_{k}^{N}(X_s^{N,i}) \right|^2\d s \bigg)^{\frac{m(d+1)}{2}} \Bigg]\\ \leq&\,C_{m,d}\left|T \, \mathrm{Tr}\left(Q_N(0)\right)\right|^{\frac{m(d+1)}{2}} \leq\, C_{m,\nu,d,T}.
	\end{aligned}
	$$
	We finish the proof by combining the above estimates.
\end{proof}

\subsection{Boltzmann entropy}\label{sectionentropy}

The Boltzmann entropy $H_N(F)$ of a probability density function $F$ on $\R^{dN}$ is defined as
\begin{equation*}
	H_{N}(F)=\frac{1}{N} \int_{\R^{dN}} F(X) \log F(X) \mathrm{d} X,
\end{equation*}
where $\d X=\d x_1\ldots \d x_N$ is the Lebesgue measure on $\R^{dN}$. The renormalization by the factor $\frac{1}{N}$ is made so that
\begin{equation}\label{iidotimes}
	H_N(f^{\otimes N})=H_1(f),
\end{equation}
where $f^{\otimes N}(X)=f(x_1)\cdots f(x_N)$. The Boltzmann entropy $H_N(F)$ has the following property (see \cite[Proposition 1]{JW17}): if $F$ is invariant under permutations of its variables and  $F^{(2)}$ denotes the marginal distribution of $F$ on $\R^{2d}$, then
\begin{equation}\label{F2}
	H_{N}(F) \geq H_{2}\big(F^{(2)}\big), \quad \text { for all } N \geq 2 .
\end{equation}

Let $F^N_t$ be the density function on $\R^{dN}$ of the particles  $X^N_t=\big(X_{t}^{N,1},\ldots ,X_{t}^{N,N}\big)$ associated to \eqref{IPS}; its existence will be shown shortly in Lemma \ref{density} below. For $X=(x_1,\ldots ,x_N)\in\R^{dN}$, $F_0^N(X)=\omega_0(x_1)\cdots \omega_0(x_N)$, where $\omega_0$ is the probability density of $X_{0}^{N,i}= X^i_0$. We want to prove an estimate on the Boltzmann entropy $H_N(F_t^N)$, for which we need some notations. Define the diffusion vector fields on $\R^{dN}$:
\begin{equation*}
	A_k^N(X)=\big(\sigma_{k}^N(x_1),\ldots ,\sigma_{k}^N(x_N)\big)\in\R^{dN},
\end{equation*}
and the drift field $A_{\epsilon}^N(X)=\big(A_{\epsilon}^{N,1}(X),\ldots ,A_{\epsilon}^{N,N}(X)\big)$ is defined as
\begin{equation*}
	A_{\epsilon}^{N,i}(X)=\frac{1}{N} \sum_{j=1}^{N} (K_{\epsilon}\ast V^N)\left(x_{i}-x_{j}\right), \quad 1\leq i\leq N.
\end{equation*}
Then the system \eqref{IPS} can be simply written as
\begin{equation*}
	\mathrm{d} X_{t}^{N}=A_{\epsilon}^N\left(X_{t}^{N}\right) \mathrm{d} t+\sum_{k}  A^N_{k}\left(X_{t}^{N}\right) \mathrm{d} W_{t}^{k}.
\end{equation*}

\begin{lemma}\label{density}
Under our assumptions, if the initial density $F^N_0\in C_c(\R^{dN})$, then for any $p>1$ and $T>0$, there exists $C_{p,T}>0$ such that for all $t\in [0,T]$, it holds
  $$F^N_t(X) \leq \|F^N_0\|_\infty C_{p,T} |X|^{-p} \quad \mbox{as } |X|\to \infty. $$
Here $|\cdot|$ is the norm in $\R^{dN}$
\end{lemma}

\begin{proof}
	The above SDE defines a stochastic flow of diffeomorphisms $\{\phi_t^N\}_{t\geq0}$, so that $X^N_t= \phi^N_t(X^N_0)$. Let $\mathcal{L}^{dN}$ be the Lebesgue measure on $\R^{dN}$, then we have
	$$
	(\phi_t^N)_{\#}\mathcal{L}^{dN}=\rho_t^N \mathcal{L}^{dN}\quad \text{and}  \quad (\phi_t^N)^{-1}_{\#}\mathcal{L}^d=\tilde{\rho}_t^N\mathcal{L}^{dN},
	$$
	where $\rho_t^N$ and $\tilde{\rho}_t^N$ are the Radon-Nikodym derivatives; moreover, by \cite[Lemma 4.3.1]{Kunita90}, $\tilde{\rho}_t^N$ admits the explicit expression ($\nabla_{N}=\left(\nabla_{x_1},\ldots ,\nabla_{x_N}\right)$ is the gradient operator on $\R^{dN}$)
	$$
	\tilde{\rho}_t^N(X)=\exp \bigg(\int_{0}^{t}\big(\nabla_N\cdot A_{\epsilon}^N\big)(\phi_s^N(X))\, \d s+ \sum_{k}  \int_{0}^{t}\big(\nabla_N\cdot A_{k}^N\big)(\phi_s^N(X))\circ\d W_s^k\bigg).
	$$
	Since the vector fields $\{A_k^N\}_k$ are divergence free and $\nabla_N\cdot A_{\epsilon}^N \geq 0$ as shown in the proof of Lemma \ref{upperbound} below, one has
	$$
	\tilde{\rho}_t^N(X)=\exp \bigg(\int_{0}^{t} \big(\nabla_N\cdot A_{\epsilon}^N\big)(\phi_s^N(X))\,\d s \bigg)\geq 1,
	$$
	and hence $	\rho_t^N(X)=\big[\tilde{\rho}_t^N\big((\phi_t^N)^{-1}(X)\big) \big]^{-1} \leq 1$. For any test function $f\in C_c(\R^{dN})$,
	\begin{equation*}
		\begin{aligned}
			\E f(X_t^N)& =\E f\big(\phi_t^N(X_0^N)\big) =\E\left[\E\left(f\left(\phi_t^N(X_0^N)\right) \big|\F_0\right)\right]\\
			&= \E\left[\E\left(f(\phi_t^N(X))\right)_{X=X_0^N} \right]=\int_{\R^{dN}}\E\big(f(\phi_t^N(X))\big) F_0^N(X)\, \d X\\
			&=\E\left[\int_{\R^{dN}} f(\phi_t^N(X)) F_0^N(X)\, \d X\right],
		\end{aligned}
	\end{equation*}
	by changing the variable,
	\begin{equation*}
		\begin{aligned}
			\E f(X_t^N) &= \E\left[\int_{\R^{dN}} f\left(Y\right)F_0^N\big((\phi_t^N)^{-1}(Y)\big) \rho_t^N(Y)\, \d Y\right]\\
			&= \int_{\R^{dN}} f(Y) \E\big[F_0^N\big((\phi_t^N)^{-1}(Y)\big) \rho_t^N(Y)\big]\, \d Y.
		\end{aligned}
	\end{equation*}
	Thus the density function of $X_t^N$ is $F_t^N(Y)= \E\big[F_0^N\big((\phi_t^N)^{-1}(Y)\big) \rho_t^N(Y)\big]$. We have $F_t^N(Y) \leq \E\big[ F_0^N\big((\phi_t^N)^{-1}(Y)\big) \big]$ since $\rho_t^N\leq 1$.
	
	Now we are ready to prove the desired assertion. In fact, fix any $\theta>1$, by \cite[Exercise 4.5.9]{Kunita90}, there exists a positive random variable $\xi$ with finite moments of any order, such that $|\phi^N_t(X)| \leq \xi (1+|X|)^\theta$ for any $X\in \R^{dN}$ and $t\in [0,T]$. Hence,
	$$|X|= \big|\phi^N_t\big((\phi^N_t)^{-1}(X)\big) \big| \leq \xi \big(1+ \big|(\phi^N_t)^{-1}(X)\big|\big)^\theta ,$$
	which implies $(\xi^{-1} |X|)^{1/\theta} \leq 1+ \big|(\phi^N_t)^{-1}(X)\big| $. Next, assume $F^N_0$ is supported in the ball $B(R_N) \subset \R^{dN}$, then
	$$\aligned
	F_t^N(X) &\leq \E\big[F_0^N\big((\phi_t^N)^{-1}(X) \big)\big] \leq \|F^N_0\|_\infty\, \P\big(\big|(\phi^N_t)^{-1}(X)\big| \leq R_N \big) \\
	&\leq \|F^N_0\|_\infty\, \P\big((\xi^{-1} |X|)^{1/\theta} \leq 1+ R_N \big) = \|F^N_0\|_\infty\, \P\big(\xi \geq |X|(1+R_N)^{-\theta} \big),
	\endaligned $$
	By Chebyshev's inequality, for any $p>1$,
	$$F_t^N(X) \leq \|F^N_0\|_\infty\, \E (\xi^p) \frac{(1+R_N)^{p\theta}}{|X|^p} \quad \mbox{as } |X|\to \infty $$
and the proof is complete.
\end{proof}

For any $ \Phi \in C^{2} (\R^{dN})$, the infinitesimal generator associated to $\{X^N_t\}_{t\geq 0}$ has the form:
\begin{equation*}
	\mathcal{L}_{\epsilon} \Phi(X)=\frac{1}{2} \sum_{k}\left\langle A_{k}^N, \nabla_{N}\left\langle A_{k}^N, \nabla_{N} \Phi\right\rangle_{\R^{dN}}\right\rangle_{\R^{dN}}+\left\langle A_{\epsilon}^N, \nabla_{N} \Phi\right\rangle_{\R^{dN}}.
\end{equation*}
Thanks to Hypotheses \ref{hypothesis-initial}, we can easily prove that $H_1(\omega_0)$ is finite. Indeed, we have
$$
\int_{\R^d} \omega_{0} \log \omega_{0}\, \d x=\int_{\{\omega_{0}\leq1\}} \omega_{0} \log \omega_{0}\, \d x+\int_{\{\omega_{0}>1\}} \omega_{0} \log \omega_{0}\, \d x.
$$
The $ L^p$-integrability of $\omega_{0}$ implies the positive part $\int_{\{\omega_{0}>1\}} \omega_{0} \log \omega_{0}\, \d x$ is bounded from above; moreover, since $\omega_0$ has finite moment of order $m(d+1)$, a similar argument as in Lemma \ref{lowerbound} below shows that the negative part $\int_{\{\omega_{0}\leq1\}} \omega_{0} \log \omega_{0}\, \d x$ has a lower bound. The following lemma gives an upper bound of Boltzmann entropy.

\begin{lemma}\label{upperbound}
	Assume Hypotheses \ref{hypothesis0}, \ref{hypothesis} and \ref{hypothesis-initial}; for all $t>0$,
	\begin{equation*}
		H_N(F_t^N)\leq H_1(\omega_0).%+tC_0.
	\end{equation*}
\end{lemma}

\begin{proof}
	The computations below are a little formal. To be more rigorous, we should approximate $\omega_0$ by a sequence of smooth and compactly supported functions $\{\omega_{0,n} \}_{n\geq0}$ in $L^p$-norm, and finally use the lower-semicontinuity of $H_N$ with respect to the weak convergence of probability measures. By  Lemma \ref{density}, $F^N_t$ is a rapidly decreasing function when $F^N_0 \in C_c^\infty(\R^{dN})$; this partly justifies the integration by parts used below.
	
	Employing the divergence free property of $A_k^N$, the density function $F_t^N$ satisfies the Fokker-Planck equation:
	\begin{equation*}
		\begin{aligned}
			\partial_{t} F_{t}^{N}&\,=\mathcal{L}^{\ast }_{\epsilon} F_{t}^{N}=\frac{1}{2} \sum_{k}\left\langle A_{k}^N, \nabla_{N}\left\langle A^N_{k}, \nabla_{N} F_{t}^{N}\right\rangle_{\R^{dN}}\right\rangle_{\R^{dN}}- \nabla_{N} \cdot\left(F_{t}^{N}A_{\epsilon}^N \right).
		\end{aligned}
	\end{equation*}
	Therefore,
	\begin{equation*}
		\begin{aligned}
			&\partial_{t}\left(F_{t}^{N} \log F_{t}^{N}\right)=\,\left(1+\log F_{t}^{N}\right) \partial_{t} F_{t}^{N} \\
			=&\, \frac{1}{2} \sum_{k}\left(1+\log F_{t}^{N}\right)\left\langle A_{k}^N, \nabla_{N}\left\langle A_{k}^N, \nabla_{N} F_{t}^{N}\right\rangle_{\R^{dN}}\right\rangle_{\R^{dN}}
			%\\&\,\quad
			-\left(1+\log F_{t}^{N}\right) \nabla_{N} \cdot\left(F_{t}^{N}\! A_{\epsilon}^N \right).
		\end{aligned}
	\end{equation*}
	Then,  by the definition of $H_{N}(F_{t}^{N})$ and integration by parts,
	{\small \begin{equation*}
			\begin{aligned}
				\frac{\mathrm{d}}{\mathrm{d} t} H_{N}(F_{t}^{N}) &\,=\frac{1}{N} \int_{\R^{dN}}\big(1+\log F_{t}^{N}\big) \partial_{t} F_{t}^{N} \mathrm{d} X \\
				&\,=-\frac{1}{2N}\! \sum_{k}\! \int_{\R^{dN}}\! \frac{\left\langle A_{k}^N, \nabla_{N} F_{t}^{N}\right\rangle_{\R^{dN}}^{2}}{F_{t}^{N}}\,\d X +\frac{1}{N}\! \int_{\R^{dN}}\! \big\< F_{t}^{N}\! A_{\epsilon}^N,\nabla_{N} \log F_{t}^{N} \big\>_{\R^{dN}}\, \d X.\\
			\end{aligned}
	\end{equation*} }
	Let $I_{\eps}^N$ be the last quantity on the right-hand side; then
	\begin{equation*}
		\begin{aligned}
			I_{\eps}^N=\,\frac{1}{N} \int_{\R^{dN}} \big\< A_{\epsilon}^N,\nabla_{N}F_{t}^{N}  \big\>_{\R^{dN}}\d X=\,-\frac{1}{N} \int_{\R^{dN}}(\nabla_{N}\cdot A_{\epsilon}^N) F_{t}^{N}\d X. \\
		\end{aligned}
	\end{equation*}
	By Hypotheses \ref{hypothesis0} and \ref{hypothesis}, we have $\nabla\cdot K\geq0$, $\rho_{\eps}\geq0$ and $V^N\geq0$, then
	\begin{equation*}
		\begin{aligned}
			\nabla_{N}\cdot A_{\epsilon}^N(X)=&\,\sum_{i=1}^{N}\frac{1}{N}\sum_{j=1}^{N}\nabla\cdot(K_{\eps}\ast V^N)(x_i-x_j)\\
			=&\,\frac{1}{N}\sum_{i,j=1}^{N} \big((\nabla\cdot K)\ast \rho_{\eps}\ast V^N \big)(x_i-x_j)\geq 0.\\
		\end{aligned}
	\end{equation*}
	Then, we have
	\begin{equation*}
		\begin{aligned}
			\frac{\mathrm{d}}{\mathrm{d} t} H_{N}(F_{t}^{N})
			&\,=-\frac{1}{2N} \sum_{k} \int_{\R^{dN}} \frac{\left\langle A_{k}^N, \nabla_{N} F_{t}^{N}\right\rangle_{\R^{dN}}^{2}}{F_{t}^{N}} \mathrm{~d} X+I_{\eps}^N\leq\,0.
		\end{aligned}
	\end{equation*}
	By Hypotheses \ref{hypothesis-initial} and \eqref{iidotimes}, we obtain $H_N(F_t^N)\leq H_N(F_0^N)= H_1(\omega_0)$ for all $t>0$.
\end{proof}

The following lemma is inspired by \cite[Section 3.1]{HM14}. We prove that the entropy of a joint probability density $F$ on $\R^{2d}$,  with finite $k$-th  moment, has a uniform lower bound on any subset $B\subset\R^{2d}$; this is crucial for us to prove the estimate on covariance in Section \ref{sectioncovariance}.

\begin{lemma}\label{lowerbound}
	Let us fix $k>0$, $F$ be a probability density on $\R^{2d}$ with finite $k$-th moment. Then, for any subset $B\subset \R^{2d}$,
	\begin{align*}
		\int_{B}\left(F\log F\right)(x,y)\d x\d y\geq C	(k)-\int_{\R^{2d}}F(x,y)\big(|x|^k+|y|^k\big)\d x\d y,
	\end{align*}
	where $C(k)$ is a constant.
\end{lemma}

\begin{proof}
	Define the function	$G_k(x,y):=c_k{\rm exp}\left(-|x|^k-|y|^k\right)$, where $c_k$ is chosen such that $G_k$ is a probability density in $\R^{2d}$. Let $h(z):=z\log z-z+1$; it is easy to see that $h(z)\geq0$ for any $z\in(0,\infty)$. Then,
	\begin{equation*}
		\begin{aligned}
			\int_{B}\left(F\log F\right)(x,y)\d x\d y=&\,\int_{B}\Big[h\Big(\frac{F}{G_k}\Big)G_k+F-G_k+F\log G_k\Big](x,y)\d x\d y\\
			\geq&\, -1+\int_{B}F(x,y)\big(\log c_k-|x|^k-|y|^k\big)\d x\d y\\
			\geq&\, C(k)-\int_{\R^{2d}}F(x,y)\big(|x|^k+|y|^k\big)\d x\d y.
		\end{aligned}
	\end{equation*}
\end{proof}

%\section{Approximation}

\section{Estimate on the covariance of random forces}\label{sectioncovariance}

Recall the particle system \eqref{IPS} with environmental noise; we can regard $Q_N\big( X_{s}^{N,i} -X_{s}^{N,j} \big)$ as the covariance of random forces on two particles $X_s^{N,i},\, X_s^{N,j}\ (1\leq i\neq j\leq N)$. In this section, we give an estimate on $\big| Q_N\big( X_{s}^{N,i} -X_{s}^{N,j} \big) \big|$ which will play an important role in proving bounds on the stochastic convolution \eqref{mollifiedzt}, see the proof of Theorem \ref{stochasticconvolution}.

The following lemma asserts that the covariance of random forces on any two distinct particles vanishes as the number of particles goes to $\infty$. Roughly speaking, forces become more and more independent when the number of particles increases. After choosing some $N$-related parameters in our proof, we can compute the convergence rate. We follow some ideas in the proof of \cite[Proposition 3.3]{FL21} and assume $\epsilon$ is fixed in this section.

\begin{lemma}\label{convratelemma}
	Under Hypotheses \ref{hypothesis0}, \ref{hypothesis} and \ref{hypothesis-initial}, for any two different particles $X_{\cdot}^{N,i}$ and $X_{\cdot}^{N,j}$ $(i\neq j)$, for some $\ell\geq2$ and any $s\in[0,T]$, we have the following estimate:
	\begin{equation}\label{covconv}
		\E\left[\big| Q_N\big( X_{s}^{N,i} -X_{s}^{N,j} \big) \big|^{\ell}\right]\leq C\big( 1+\left\|K_{\epsilon}\right\|_{\infty}^{2d} \big) N^{-\frac{1}{2}},
	\end{equation}
	where $C$ is a constant depending on $\omega_{0},\ell,\nu,d$ and $T$. Note that when $i=j$, $\big| Q_N\big( X_{s}^{N,i} -X_{s}^{N,j} \big) \big| = |Q_N(0) |= |2\nu I_d |=2\nu\sqrt{d}$ is a constant.
\end{lemma}
\begin{proof}
	%	If $i=j$, $\left|Q_N(X_{s}^{N,i}-X_{s}^{N,j})\right|=\left|Q_N(0)\right|=4\nu$.
	By symmetry, without loss of generality, we can take $i=1$ and $j=2$. For any ball $B(R)\subset\R^d$ centered at origin with radius $R$, we have	
	\begin{equation*}
		\begin{aligned}
			\E\left[\left|Q_N(X_{s}^{N,1}-X_{s}^{N,2})\right|^{\ell}%
			\right] =&\,\int_{(B(R)^2)^c}\left|Q_N(x_1-x_2)\right|^{\ell}%\phi(x_1)\phi(x_2)
			F_s^{2,N}(x_1,x_2)\d x_1\d x_2 \\
			&+\int_{B(R)^2}\left|Q_N(x_1-x_2)\right|^{\ell}%\phi(x_1)\phi(x_2)
			F_s^{2,N}(x_1,x_2)\d x_1\d x_2 \\
			=:&\,J_N^{1}+J_N^{2},
		\end{aligned}
	\end{equation*}
	where $F_s^{2,N}$ is the joint density function of $\big(X_s^{N,1},X_s^{N,2}\big)$ and $B(R)^2=B(R)\times B(R)$.	For the first term $J_N^1$, we have $|Q_N(z)| \leq \mbox{Tr}(Q_N(0)) =2\nu d$ for any $z\in \R^d$;  by Chebyshev's inequality, one has
	\begin{equation*}
		\begin{aligned}
			J_N^{1}\leq&\,C_{\ell,\nu}\int_{\{|x_1|\vee|x_2|\geq R\}}F_s^{2,N}(x_1,x_2)\d x_1\d x_2  \\ \leq&\,C_{\ell,\nu}\int_{\R^{2d}}\frac{\left|x_1\right|^{2d}+\left|x_2\right|^{2d}}{R^{2d}}F_s^{2,N}(x_1,x_2)\d x_1\d x_2 \\ \leq&\,\frac{C_{\ell,\nu}}{R^{2d}}\left(\E\left|X_{s}^{N,1}\right|^{2d}+\E\left|X_{s}^{N,2}\right|^{2d}\right).
		\end{aligned}
	\end{equation*}
	Combining with Lemma \ref{1stmoment}, we deduce that
	\begin{equation}\label{2dmoment}
		\E\big|X_{s}^{N,i}\big|^{2d}\leq C_{\nu,d,T}\big(1+\left\|K_{\epsilon}\right\|_{\infty}^{2d}  \big),
	\end{equation}
	and then
	\begin{equation}\label{j1n}
		\begin{aligned}			
			J_N^{1}	\leq\,\frac{C_{\ell,\nu,d,T}}{R^{2d}} \big( 1+\left\|K_{\epsilon}\right\|_{\infty}^{2d} \big).
		\end{aligned}
	\end{equation}
	
	For the second term $J_N^2$, following the method used in the proof of \cite[Proposition 3.3]{FL21}, for some $M>1$, we have
	\begin{equation*}
		\begin{aligned}
			J_N^2		=&\,\int_{B(R)^2\cap\{F_s^{2,N}<M\} }\left|Q_N(x_1-x_2)\right|^{\ell}
			F_s^{2,N}(x_1,x_2)\d x_1\d x_2  \\
           &\,+ \int_{B(R)^2\cap\{F_s^{2,N}\geq M\}}\left|Q_N(x_1-x_2)\right|^{\ell}%\phi(x_1)\phi(x_2)
			F_s^{2,N}(x_1,x_2)\d x_1\d x_2 \\
			=:&\,J_N^{2,1}+J_N^{2,2}.
		\end{aligned}
	\end{equation*}	
	For $J_N^{2,1}$, by Hypotheses \ref{hypothesis},  %with volume $\left|B(R)\right|$, we have the following estimate:
	%\begin{equation}
	%	\begin{aligned}
		%		\,\int_{B(R)\times B(R)}\left|Q_{N}\left(x-y\right)\right|^{r} \d x\d y
		%		=&\,\int_{B(R)}\d y\int_{B(R)}\left|Q_{N}\left(x-y\right)\right|^{r} \d x\\
		%		\leq&\,\left|B(R)\right|\int_{\R^d}\left|Q_N(x)\right|^{r} \d x
		%		\leq&\,\frac{\left|B(R)\right|}{e^{dN}}\int_{\R^d}\left|\left(\F^{-1}f\right)(x^{\prime})\right|^{q} \d x^{\prime}
		%		\leq\,\frac{C_rR^d}{e^{dN}}.
		%	\end{aligned}
	%\end{equation}
	it holds
	\begin{equation}\label{j21n}
		\begin{aligned}
			J_N^{2,1}\leq&\, M\int_{B(R)^2}\left|Q_N(x_1-x_2)\right|^{\ell}\d x_1\d x_2\\
			=&\,M\int_{B(R)}\d x_1\int_{B(R)}\left|Q_N(x_1-x_2)\right|^{\ell}\d x_2\\ \leq&\,C_dMR^d\left\|Q_N\right\|_{L^{{\ell}}}^{\ell}\leq\,\frac{C_{\ell,d}MR^d}{e^{2N}}.
		\end{aligned}
	\end{equation}
	And for $J_N^{2,2}$, we have
	\begin{equation*}
		\begin{aligned}
			J_N^{2,2}\leq&\,C_{\ell,\nu}\int_{B(R)^2\cap\{F_s^{2,N}\geq M\}}%\phi(x_1)\phi(x_2)
			F_s^{2,N}(x_1,x_2)\d x_1\d x_2  \\
			\leq&\,\frac{C_{\ell,\nu}}{\log  M}\int_{B(R)^2\cap\{F_s^{2,N}\geq 1\}}\left(F_s^{2,N}\log F_s^{2,N}\right)(x_1,x_2)\d x_1\d x_2 .
			\\
		\end{aligned}
	\end{equation*}
	Recalling the fact that $z\log z\in[-e^{-1},0]$ for all $z\in\left(0,1\right]$, then
	\begin{equation*}
		\begin{aligned}
			J_N^{2,2}\leq&\,\frac{C_{\ell,\nu}}{\log  M}\bigg[\int_{B(R)^2}\left(F_s^{2,N}\log F_s^{2,N}\right)(x_1,x_2)\d x_1\d x_2+e^{-1}\left|B(R)^2\right|\bigg] \\
			\leq&\,\frac{C_{\ell,\nu,d}R^{2d}}{\log M}+\frac{C_{\ell,\nu}}{\log  M}\bigg[\int_{\R^{2d}}\left(F_s^{2,N}\log F_s^{2,N}\right)(x_1,x_2)\d x_1\d x_2\\
           &\,\qquad\qquad\qquad\qquad-\int_{(B(R)^2)^c}\left(F_s^{2,N}\log F_s^{2,N}\right)(x_1,x_2)\d x_1\d x_2\bigg].\\
		\end{aligned}
	\end{equation*}
	By \eqref{F2} and Lemma \ref{upperbound}, one has
	\begin{equation*}
		\begin{aligned}
			\int_{\R^{2d}}\left(F_s^{2,N}\log F_s^{2,N}\right)(x_1,x_2)\d x_1\d x_2&\,=2H_2(F_s^{2,N})\leq 2H_N(F^N_s)\leq 2H_1(\omega_0).
		\end{aligned}
	\end{equation*}
	Applying Lemma \ref{lowerbound} with $k=2d$ and \eqref{2dmoment}, it holds
	\begin{equation*}
		\begin{aligned}
			& -\int_{(B(R)^2)^c}\left(F_s^{2,N}\log F_s^{2,N}\right)(x_1,x_2)\d x_1\d x_2\\	
			\leq& -C(2d)+\int_{\R^{2d}}F_s^{2,N}(x_1,x_2) \big( |x_1|^{2d}+|x_2|^{2d} \big)\d x_1\d x_2\\
			\leq& -C(2d)+\E\left|X^{N,1}_s\right|^{2d} +\E\left|X^{N,2}_s\right|^{2d}
			\leq C_{\ell,\nu,d,T} \big(1+\left\|K_{\epsilon} \right\|_{\infty}^{2d} \big).
		\end{aligned}
	\end{equation*}
	Thus,	we get the estimate
	\begin{equation}\label{j22n}
		\begin{aligned}
			J_N^{2,2}
			\leq&\,\frac{C_{\ell,\nu,d}R^{2d}}{\log M}+\frac{C_{\ell,\nu,d,T} \big(1+H_1(\omega_0)+\left\|K_{\epsilon}\right\|_{\infty}^{2d} \big)}{\log  M}.
		\end{aligned}
	\end{equation}
	Putting together \eqref{j1n}, \eqref{j21n} and \eqref{j22n}, we deduce that
	\begin{equation}\label{sum}
		\begin{aligned}
			\E\left[\left|Q_N(X_{s}^{N,1}-X_{s}^{N,2})\right|^{\ell} \right]
        &\leq \frac{C_{\ell,\nu,d,T}\big( 1+\left\|K_{\epsilon}\right\|_{\infty}^{2d} \big)}{R^{2d}}+\frac{C_{\ell,d}MR^d}{e^{2N}} +\frac{C_{\ell,\nu,d}R^{2d}}{\log M} \\
        &\quad +\frac{C_{\ell,\nu,d,T} \big(1+H_1(\omega_0)+\left\|K_{\epsilon}\right\|_{\infty}^{2d} \big)}{\log  M}.
		\end{aligned}
	\end{equation}
	We finish the proof by taking $R=R(N)=N^{\frac{1}{4d}}$, $M=M(N)=e^N$.
\end{proof}

\section{Convergence of the mollified empirical measures}\label{sec-proof-main-thm}

The purpose of this part is to prove the main result. We first establish in Section \ref{sectionmollifiedmeasure} an $L^p$-bound on the mollified empirical measure $\omega^N$ defined in \eqref{mollifiedmeasure}, and then we provide the proof of Theorem \ref{thm-main} in Section \ref{sectionproof}.

\subsection{Bounds on the mollified empirical measures}\label{sectionmollifiedmeasure}

For any test function $\phi\in C^2(\R^d)$, by \eqref{compactform} and It\^{o} formula,
\begin{equation*}%\label{F}
	\begin{aligned}
		\d \phi(\X)
		=&\, \big(K_{\epsilon}\ast V^N \ast S^N_t \big)( X_{t}^{N,i}) \cdot\nabla\phi(\X)\d t +\sum_{k} \sigma_{k}^{N}(X_{t}^{N,i}) \cdot\nabla\phi(\X) \d W_t^k\\
		&\, +\frac{1}{2}\sum_{k}\mathrm{Tr}\left[\left(\sigma_{k}^{N}\otimes\sigma_{k}^{N}\right)\nabla^2\phi\right](X_{t}^{N,i})\d t,
	\end{aligned}
\end{equation*}
where
$$
\sum_{k}\mathrm{Tr} \left[\left(\sigma_{k}^{N}\otimes\sigma_{k}^{N}\right) \nabla^2\phi\right](X_{t}^{N,i}) =\mathrm{Tr}\big[Q_N(0) \nabla^2\phi(X_{t}^{N,i}) \big]=2\nu\Delta\phi(X_{t}^{N,i}).
$$
Then, the empirical measure $S_t^N$ satisfies
\begin{equation}\label{F}
	\begin{aligned}
		\<S_t^N,\phi\>
		&=\<S_0^N,\phi\> +\int_{0}^{t}\<S_s^N,(K_{\epsilon}\ast V^{N}\ast S_s^N)\cdot\nabla\phi\>\d s\\
		&\quad\, +\nu\int_{0}^{t}\<S_s^N,\Delta\phi\>\d s+
		\sum_{k}\int_{0}^{t}\<S_s^N,\sigma_k^N\cdot\nabla\phi\>\d W_s^k.
	\end{aligned}
\end{equation}
For any $x\in \R^d$, we shall take the test function $\phi_x(y)=V^N(x-y)$ in identity \eqref{F};  notice that
$$\nabla\phi_x(y)=\nabla_yV^N(x-y)=-\nabla_x V^N(x-y).$$
Recall \eqref{mollifiedmeasure} for the mollified empirical measure; then
\begin{equation*}
	\begin{aligned}
		\omega_{t}^{N}(x)=&\ \omega_{0}^{N}(x) -\int_{0}^{t}\left\langle S_{s}^{N}, \left(K_{\epsilon} \ast  \omega_{s}^{N}\right) \cdot \nabla_x V^{N}(x-\cdot)\right\rangle \d s \\
		&\,+\nu \int_{0}^{t} \Delta \omega_{s}^{N}(x) \d s- \sum_{k} \int_{0}^{t} \big\< S_s^N,\sigma_k^N\cdot\nabla_x V^{N} (x-\cdot) \big\>\,  \d W_{s}^{k}.
	\end{aligned}
\end{equation*}
For the sake of clarity, we write:
\begin{equation*}
	\begin{aligned}
		\left(\nabla V^{N} \ast \left(\left(K_{\epsilon} \ast  \omega_{s}^{N}\right) S_{s}^{N}\right)\right)(x)&\,:=\left\langle S_{s}^{N}, \left(K_{\epsilon} \ast  \omega_{s}^{N}\right) \cdot \nabla_x V^{N}(x-\cdot)\right\rangle;\\
		\left(\nabla V^{N} \ast \left(\sigma_k^N S_{s}^{N}\right)\right)(x)&\,:=\left\langle S_{s}^{N}, \sigma_k^N \cdot \nabla_x V^{N}(x-\cdot)\right\rangle.
	\end{aligned}
\end{equation*}
Recall the operator  $A=\nu\Delta$ defined in Section \ref{notationsandlemma}; as in \cite{FOS20, ORT20b}, we rewrite the above equation for mollified empirical measures in mild form:
\begin{equation*}
	\begin{aligned}
		\omega_{t}^{N}=e^{t A} \omega_{0}^{N} &\,-\int_{0}^{t}  e^{(t-s) A}\left(\nabla V^{N} \ast \left(\left(K_{\epsilon} \ast  \omega_{s}^{N}\right) S_{s}^{N}\right)\right) \d s \\
		&\,-\sum_{k} \int_{0}^{t} e^{(t-s) A}\left(\nabla V^{N} \ast \left(\sigma_k^N S_{s}^{N}\right)\right)  \d W_{s}^{k}.
	\end{aligned}
\end{equation*}
By the explicit formula \eqref{semigroup} of semigroup, for any function $f\in H^{1,p}(\R^d)$, one has
\begin{equation}\label{nablaf}
	e^{(t-s) A} \nabla f=\nabla e^{(t-s) A} f,
\end{equation}
then we obtain
\begin{equation}\label{mollified}
	\begin{aligned}
		\omega_{t}^{N}=e^{t A} \omega_{0}^{N} &\,-\int_{0}^{t} \nabla\cdot e^{(t-s) A}\left( V^{N} \ast \left(\left(K_{\epsilon} \ast  \omega_{s}^{N}\right) S_{s}^{N}\right)\right) \d s -Z_t^N,
	\end{aligned}
\end{equation}
where
\begin{equation}\label{mollifiedzt}
	Z^N_t=\sum_{k} \int_{0}^{t} \nabla\cdot e^{(t-s) A}\left( V^{N} \ast \left(\sigma_k^N S_{s}^{N}\right)\right)  \d W_{s}^{k}.
\end{equation}
Before moving forward, we present an important estimate on the stochastic convolution $Z_t^N$.

\begin{theorem}\label{stochasticconvolution}
	Assume Hypotheses \ref{hypothesis0}, \ref{hypothesis} and \ref{hypothesis-initial}.
	Then there exist a constant $C=C_{m,T} >0$ and some $\iota>0$ such that for all $N\in\N$,
	\begin{equation}\label{Niota}
		\bigg\|\sup_{t \in[0, T]} \left\| Z^N_t\right\|_{\L (\R^{d} )} \bigg\|_{L^m(\Omega)}\leq C_{m,T}\left\|K_{\epsilon}\right\|_{\infty}^{d+1}N^{-\iota}.
	\end{equation}
\end{theorem}

The proof of Theorem \ref{stochasticconvolution} is motivated by \cite[Section A.2]{ORT20b} and follows readily from the two propositions below.

\begin{proposition}\label{supztp}
	Under Hypotheses \ref{hypothesis0}, \ref{hypothesis} and \ref{hypothesis-initial}, for $T\in(0,T_{max})$, there exists $C_{m,\delta,T}>0$ such that  for any $N\in\N$ and small enough $\epsilon>0$,
	\begin{equation}\label{zp}
		\bigg\|\sup_{t \in[0, T]}\left\| Z^N_t\right\|_{L^{p} (\R^{d} )} \bigg\|_{L^m(\Omega)}\leq C_{m,\delta,T}\left\|K_{\epsilon}\right\|_{\infty}^{\frac{d}{m}}N^{\beta\big(d +2\delta-\frac{d}{p}\big)-\frac{1}{4m}},
	\end{equation}
	where $\frac{2}{m}<\delta<1$.
\end{proposition}

\begin{proposition}\label{supzt1}
	Under Hypotheses \ref{hypothesis0}, \ref{hypothesis} and \ref{hypothesis-initial}, for $T\in(0,T_{max})$, there exists $C_{m,\delta,T}>0$ such that for any $N\in\N$ and any small enough $\epsilon>0$,
	\begin{equation}\label{z1}
		\bigg\|\sup_{t \in[0, T]}\left\| Z^N_t\right\|_{L^1 (\R^{d} )} \bigg\|_{L^m(\Omega)}\leq C_{m,\delta,T}\left\|K_{\epsilon}\right\|_{\infty}^{d+1}N^{\frac{\beta}{2}(d+ \delta(d+4))-\frac{1}{4m}},
	\end{equation}
	where $\frac{2}{m}<\delta<1$.
\end{proposition}

The proofs of Propositions \ref{supztp} and \ref{supzt1} are quite  technical and make use of the crucial Lemma \ref{convratelemma}; we postpone them to the appendices (see Sections \ref{sectiona3} and \ref{sectiona4}) in order to maintain the readability of the paper.

\begin{proof}[Proof of Theorem \ref{stochasticconvolution}]
	%	Since $\delta>\frac{2}{m}\geq1-\frac{2}{p}$, we have the inequality $d+ \delta(d+4)\geq 2(d+2\delta-\frac{d}{p})$.
	Combining \eqref{zp} and \eqref{z1}, for any small enough $\epsilon>0$, we get the result
	\begin{equation*}%\label{z1p}
		\bigg\|\sup_{t \in[0, T]}\left\| Z^N_t\right\|_{\L (\R^{d} )} \bigg\|_{L^m(\Omega)}\leq C_{m,\delta,T}\left\|K_{\epsilon}\right\|_{\infty}^{d+1}N^{\frac{\beta}{2}\big((2d+4\delta-\frac{2d}{p})\vee (d+\delta(d+4)) \big)-\frac{1}{4m}}.
	\end{equation*}
	If we take any $\delta\in \big(\frac{2}{m}, 1 \big)$ and $\beta$ satisfying
	$$0<\beta\leq\frac{1}{4m(d+2)},$$
	then we can find some $\iota>0$ such that \eqref{Niota} holds.
\end{proof}
%\begin{remark}
%	For example, let us set $\beta=\frac{1}{(2m+4)d+16+\eps}$ for some small $\eps>0$, and take $\delta=\frac{2}{m}+\frac{\eps}{4m(d+4)}$. Then for small enough $\varepsilon^{\prime}$,
%	$$\iota:=\frac{1}{4m}-\varepsilon^{\prime}-\frac{\beta}{2}(d+ \delta(d+4))=\frac{\eps}{16dm^2+32m(d+4)+8m\eps}-\varepsilon^{\prime}>0,$$ and one has $\left\|\sup_{t \in[0, T]}\left\| Z^N_t\right\|_{\L\left(\R^{d}\right)}\right\|_{L^m(\Omega)}\leq CN^{-\iota}$.
%\end{remark}
\begin{remark}\label{betarange}
	The range of $\beta$ can be slightly larger.	In order that the right-hand sides of \eqref{zp} and \eqref{z1} tend to 0 as $N\to \infty$, we need to take a small $\beta$ such that $0<\beta<\frac{1}{2m\big(d+4\delta+d(1-\frac{2}{p})\vee \delta \big)}$, where $\frac{2}{m}<\delta<1$. Now for $0<\beta<\frac{1}{2md+16+2md\big( (1-\frac{2}{p})\vee\frac{2}{m}\big)},$ we can find a suitable $\delta$ such that the estimate \eqref{Niota} holds for some $\iota>0$.
\end{remark}

%\subsection{Tightness of the Mollified Empirical Measure}
Now we can establish a uniform bound of the mollified empirical measure $\omega^N$.

\begin{proposition}\label{estimategtN}
	Assume Hypotheses \ref{hypothesis0}, \ref{hypothesis} and \ref{hypothesis-initial}; for any $t\in[0,T]$,
	\begin{equation*}
		\left\|\omega_{t}^{N}\right\|_{L^m(\Omega,L^p(\R^d))} \leq C_{\omega_{0},m,T}\big(1+\left\|K_{\eps}\right\|_{\infty}^{\frac{d}{m}}\big) \exp\big(C_{\nu,T}\left\|K_{\epsilon} \right\|_{\infty}^2 \big),
	\end{equation*}
	which is uniform in $N\in\N$.
\end{proposition}

\begin{proof}
	Equation \eqref{mollified} yields that
	\begin{equation}\label{triangleineq}
		\begin{aligned}
			\left\|\omega_{t}^{N}\right\|_{L^m(\Omega,L^p)}\leq&\,\left\|e^{t A}\omega_{0}^{N}\right\|_{L^m(\Omega,L^p)}	+\left\|Z_t^N\right\|_{L^m(\Omega,L^p)}\\
			&\,+\int_{0}^{t}\left\|  \nabla \cdot e^{(t-s) A}\left( V^{N} \ast \left(\left(K_{\epsilon} \ast  \omega_{s}^{N}\right) S_{s}^{N}\right)\right) \right\|_{L^m(\Omega,L^p)}\d s.
		\end{aligned}
	\end{equation}
	By Remark \ref{rem-initial}, the first term on the right-hand side of \eqref{triangleineq} can be estimated as
	\begin{equation}\label{initialtermestimate}
		\left\|e^{t A}\omega_{0}^{N}\right\|_{L^m(\Omega,L^p)}\leq
		\left\|\omega_{0}^{N}\right\|_{L^m(\Omega,L^p)}\leq C_{\omega_{0}} .
	\end{equation}
	Next, by Proposition \ref{supztp}, we have the following estimate on the stochastic convolution:
	\begin{equation}\label{martingaltermestimate}
		\begin{aligned}
			\left\|Z_t^N\right\|_{L^m\left(\Omega,L^p\right)}\leq C_{m,T}\left\|K_{\eps}\right\|_{\infty}^{\frac{d}{m}}.
		\end{aligned}
	\end{equation}
	
	It remains to deal with the third term; one has
	\begin{equation*}
		\begin{aligned}
			&\,\int_{0}^{t}\left\| \nabla \cdot e^{(t-s) A}\left( V^{N} \ast \left(\left(K_{\epsilon} \ast  \omega_{s}^{N}\right) S_{s}^{N}\right)\right) \right\|_{L^m(\Omega,L^p)}\d s\\
			\leq &\, C\int_{0}^{t}\big\| (\mathrm{I}-A)^{\frac{1}{2}}    e^{(t-s) A}\big\|_{L^{p}\rightarrow L^{p}}\left\| V^{N} \ast \left(\left(K_{\epsilon} \ast  \omega_{s}^{N}\right) S_{s}^{N}\right) \right\|_{L^m(\Omega,L^p)}\d s.
		\end{aligned}
	\end{equation*}
	For any $x\in\R^d$,
	$$
	\left|\left(V^{N} \ast \left(\left(K_{\epsilon} \ast  \omega_{s}^{N}\right) S_{s}^{N}\right) \right) (x)\right|
	\leq \left\| K_{\epsilon}  \ast  \omega_{s}^{N}\right\|_{\infty}\left|\left(V^{N} \ast  S_{s}^{N}\right)  (x)\right|\leq\left\| K_{\epsilon}  \right\|_{\infty}\left|\omega_{s}^N (x)\right|,
	$$
	where the last inequality is due to
	\begin{equation}\label{kg}
		\left\|K_{\epsilon}\ast \omega_{s}^{N}\right\|_{\infty}=\sup_{x \in\R^d}
		\left|\int_{\R^d}K_{\epsilon}(x-y)\omega_{s}^{N}(y)\d y\right|\leq \left\|K_{\epsilon}\right\|_{\infty}\left\|\omega_{s}^{N}\right\|_{L^1}=\left\|K_{\epsilon}\right\|_{\infty}.
	\end{equation}
	Then by \eqref{lpestimate}, we have
	\begin{equation}\label{drifttermestimate}
		\begin{aligned}
			&\,\int_{0}^{t}\left\| \nabla \cdot e^{(t-s) A}\left( V^{N} \ast \left(\left(K_{\epsilon} \ast  \omega_{s}^{N}\right) S_{s}^{N}\right)\right) \right\|_{L^m(\Omega,L^p)}\d s\\
			\leq&\,C_{\nu,T}\left\|K_{\epsilon}\right\|_{\infty}\int_{0}^{t}\frac{1}{\sqrt{t-s}}\left\|\omega_{s}^{N}\right\|_{L^m(\Omega,L^p)}\d s.
		\end{aligned}
	\end{equation}
	Combining \eqref{initialtermestimate}, \eqref{martingaltermestimate} and \eqref{drifttermestimate}, we obtain
	\begin{equation*}
		\begin{aligned}
			\left\|\omega_{t}^{N}\right\|_{L^m(\Omega,L^p)}\leq C_{\omega_{0}}+C_{m,T}\left\|K_{\eps}\right\|_{\infty}^{\frac{d}{m}}+C_{\nu,T}\left\|K_{\epsilon}\right\|_{\infty}\int_{0}^{t}\frac{1}{\sqrt{t-s}}\left\|\omega_{s}^{N}\right\|_{L^m(\Omega,L^p)}\d s.
		\end{aligned}
	\end{equation*}
	Applying Lemma \ref{Gronwall} and estimate \eqref{mittagbound}, we deduce the desired estimate.
\end{proof}

\subsection{Proof of Theorem \ref{thm-main}}\label{sectionproof}

The method we use to prove our main result is inspired by \cite{ORT20b}. Recall the nonlinear Fokker-Planck equation \eqref{PDE-eps} with smooth kernel; in mild formulation it reads as
\begin{equation}\label{tildeomegat}
	\omega^{\epsilon}_{t}=e^{tA}\omega_{0}-\int_{0}^{t}e^{(t-s)A}\nabla\cdot\big((K_{\epsilon} \ast  \omega^{\epsilon}_s)\omega^{\epsilon}_s\big)\d s=e^{tA}\omega_{0}-\int_{0}^{t}\nabla\cdot e^{(t-s)A}\big((K_{\epsilon} \ast  \omega^{\epsilon}_s)\omega^{\epsilon}_s\big)\d s.
\end{equation}
The proof of Theorem \ref{thm-main} can be split into two steps, thanks to the following inequality:
{\small \begin{equation}\label{split}
		\bigg\| \sup_{t \in[0, T]} \big\|\omega_t-\omega_{t}^{N} \big\|_{\L} \bigg\|_{L^m(\Omega)}\leq\sup_{t \in[0, T]} \left\|\omega_t-\omega^{\epsilon}_{t} \right\|_{\L} + \bigg\| \sup_{t \in[0, T]} \big\|\omega^{\epsilon}_{t}-\omega_{t}^{N} \big\|_{\L} \bigg\|_{L^m(\Omega)}.
\end{equation} }
The next proposition shows the convergence of the first term on the right-hand side of \eqref{split}, and the estimate on the second term will be done in Proposition \ref{main1} at the end of this section.

\begin{proposition}\label{firstpart}
	Under Hypotheses \ref{hypothesis0}, \ref{hypothesis} and  \ref{hypothesis-initial}, for any $T\in(0,T_{max})$, the solution of \eqref{PDE-eps} converges to that of \eqref{PDE} in  $C\left([0,T];\L(\R^d)\right)$, namely
	$$
	\lim\limits_{\epsilon\rightarrow0}\sup_{t \in[0, T]}\left\|\omega_t-\omega^{\epsilon}_{t}\right\|_{\L}=0.
	$$
\end{proposition}

\begin{proof}
	The mild formulations \eqref{omegat}, \eqref{tildeomegat} and estimate \eqref{lpestimate} imply that
	\begin{equation}\label{omegaomgegaeps}	
		\begin{aligned}
			\left\|\omega_t-\omega^{\epsilon}_{t}\right\|_{\L}=&\,\left\|\int_{0}^{t}\nabla\cdot e^{(t-s) A} \big( (K\ast \omega_s) \omega_s- (K_{\epsilon}\ast \omega^{\epsilon}_{s})  \omega^{\epsilon}_{s} \big)\d s\right\|_{\L}\\\leq&\  C_{\nu,T}\int_{0}^{t}\frac{1}{\sqrt{t-s}}\left\| (K\ast \omega_s) \omega_s- (K_{\epsilon}\ast \omega^{\epsilon}_{s})  \omega^{\epsilon}_{s}\right\|_{\L}\d s.
		\end{aligned}
	\end{equation}
	For the integrand, one has
	\begin{equation*}
		\begin{aligned}
		&\, \left\| (K\ast \omega_s) \omega_s- (K_{\epsilon}\ast \omega^{\epsilon}_{s})  \omega^{\epsilon}_{s}\right\|_{\L}\\\leq&\, \left\| (K\ast \omega_s) (\omega_s-\omega^{\epsilon}_{s}) \right\|_{\L}+\left\| ((K-K_\epsilon)\ast \omega_s) \omega^{\epsilon}_{s} \right\|_{\L}\\
        &\,+\left\| (K_\epsilon*(\omega_s-\omega^{\epsilon}_{s}))\omega^{\epsilon}_{s}\right\|_{\L}\\\leq&\, \left\| K\ast \omega_s \right\|_{\infty} \left\|\omega_s-\omega^{\epsilon}_{s} \right\|_{\L}+\left\| (K-K_\epsilon)\ast \omega_s\right\|_{\infty} \left\|\omega^{\epsilon}_{s} \right\|_{\L}\\
        &\,+\left\| K_\epsilon*(\omega_s-\omega^{\epsilon}_{s})\right\|_{\infty}\left\|\omega^{\epsilon}_{s}\right\|_{\L}.\\
		\end{aligned}
	\end{equation*}
	By \eqref{l1p}, we have $\| K\ast \omega_s \|_{\infty} \leq C_K\|\omega_s \|_{L^1\cap L^p}$ and $\| K_{\eps}\ast (\omega_s- \omega^\epsilon_s) \|_{\infty} \leq C_K\|\omega_s -\omega^\epsilon_s \|_{L^1\cap L^p}$; moreover,
	$$\| (K-K_\epsilon) \ast \omega_s \|_{\infty}= \|K\ast (\omega_s - \rho_\epsilon \ast \omega_s)\|_\infty \leq C_K\|\omega_s - \rho_\epsilon \ast \omega_s\|_{L^1\cap L^p}.$$
	Therefore,
	\begin{equation}\label{komega}
		\begin{aligned}
			&\,\left\| (K\ast \omega_s) \omega_s- (K_{\epsilon}\ast \omega^{\epsilon}_{s})  \omega^{\epsilon}_{s}\right\|_{\L}\\
			\leq&\, C_K \left\|\omega_s-\omega^{\epsilon}_{s} \right\|_{\L}\left(\left\| \omega_s \right\|_{\L}+\left\| \omega^{\epsilon}_s \right\|_{\L}\right)+C_K\left\| \omega_s-\rho_{\epsilon}\ast \omega_{s}\right\|_{\L}\left\| \omega^{\epsilon}_{s} \right\|_{\L}\\
			\leq&\, C\left\|\omega_s-\omega^{\epsilon}_{s} \right\|_{\L}+C\left\| \omega_s-\rho_{\epsilon}\ast \omega_{s}\right\|_{\L},
		\end{aligned}
	\end{equation}
	where the last inequality is due to
	$$
	\left\|\omega\right\|_{C\left([0,T],\L(\R^d)\right)}\bigvee \sup_{\epsilon>0}\left\|\omega^{\epsilon}\right\|_{C\left([0,T],\L(\R^d)\right)}\leq C<\infty.
	$$
	Substituting estimate \eqref{komega} into \eqref{omegaomgegaeps} and applying Lemma \ref{Gronwall}, we deduce that
	\begin{equation}\label{boundomega}
		\begin{aligned}
			\sup_{t \in[0, T]}\left\|\omega_t-\omega^{\epsilon}_{t}\right\|_{\L}\leq&\, C_T\, \bigg(\sup_{t \in[0, T]}\int_{0}^{t} \frac{1}{\sqrt{t-s}} \left\| \omega_s-\rho_{\epsilon}\ast \omega_{s}\right\|_{\L}\d s \bigg).
		\end{aligned}
	\end{equation}		
	By H\"older's inequality, for some $q\in(1,2)$ with its conjugate number $q^{\prime}$,
	$$
	\begin{aligned}
		&\,\sup_{t \in[0, T]}\int_{0}^{t}\frac{1}{\sqrt{t-s}}	\left\| \omega_s-\rho_{\epsilon}\ast \omega_{s}\right\|_{\L}\d s\\\leq&\,\sup_{t \in[0, T]}\bigg(\int_{0}^{t}\frac{1}{(t-s)^{\frac{q}{2}}}\d s\bigg)^{\frac{1}{q}}\left(\int_{0}^{t}\left\| \omega_s-\rho_{\epsilon}\ast \omega_{s}\right\|^{q^{\prime}}_{\L}\d s\right)^{\frac{1}{q^{\prime}}}\\\leq&\,C_{T,q}\left(\int_{0}^{T}\left\| \omega_s-\rho_{\epsilon}\ast \omega_{s}\right\|^{q^{\prime}}_{\L}\d s\right)^{\frac{1}{q^{\prime}}}.
	\end{aligned}
	$$
	It is clear that
	$
	\lim\limits_{\epsilon\rightarrow0}	\left\| \omega_s-\rho_{\epsilon}\ast \omega_{s}\right\|_{\L}=0
	$
	for any $s\in[0,T]$, by dominated convergence theorem, we deduce that
	$$
	\lim\limits_{\epsilon\rightarrow0}\int_{0}^{T}\left\| \omega_s-\rho_{\epsilon}\ast \omega_{s}\right\|^{q^{\prime}}_{\L}\d s=0.
	$$
	Combining these results with \eqref{boundomega}, we finish the proof.
\end{proof}

Now we aim to estimate the second part on the right-hand side of \eqref{split}.
Recall the mild formulations \eqref{mollified} and \eqref{tildeomegat}; we have
\begin{equation}\label{estimate}
	\begin{aligned}
		\left\|\omega^{\epsilon}_{t}-\omega_{t}^{N}\right\|_{\L}\leq&\,\left\|e^{tA}\left(\omega_0-\omega_{0}^{N}\right)\right\|_{\L}+\left\| Z^N_t\right\|_{\L}\\
		&+\int_{0}^{t}\left\| \nabla\cdot e^{(t-s) A}\left( (K_\epsilon*\omega^{\epsilon}_{s}) \omega^{\epsilon}_{s}- V^{N}*\left( (K_{\epsilon}\ast \omega^N_s)S_s^N\right)\right)\right\|_{\L}\d s.
	\end{aligned}
\end{equation}
The third term on the right-hand side of \eqref{estimate} can be estimated as follows: by \eqref{lpestimate},
\begin{equation}\label{integrandsecond}
	\begin{aligned}
		&\,\int_{0}^{t}\left\| \nabla\cdot e^{(t-s) A}\left( (K_\epsilon*\omega^{\epsilon}_{s}) \omega^{\epsilon}_{s}- V^{N}*\left( (K_{\epsilon}\ast \omega^N_s)S_s^N\right)\right)\right\|_{\L}\d s\\
		\leq&\,C_{\nu,T}\int_{0}^{t} \frac{1}{\sqrt{t-s}}\left\| (K_\epsilon*\omega^{\epsilon}_{s}) \omega^{\epsilon}_{s}- V^{N}*\left( (K_{\epsilon}\ast \omega^N_s)S_s^N\right)\right\|_{\L}\d s.\\
	\end{aligned}
\end{equation}
We observe that
\begin{equation*}
	\begin{aligned}
		& \left\| (K_\epsilon*\omega^{\epsilon}_{s}) \omega^{\epsilon}_{s}- V^{N}\ast\left( (K_{\epsilon}\ast \omega^N_s)S_s^N\right) \right\|_{\L} \\
		\leq& \left\| (K_{\epsilon}\ast \omega^{\epsilon}_{s}) \omega^{\epsilon}_{s}- (K_{\epsilon}\ast \omega^N_s)  \omega^N_s\right\|_{\L} +\left\| (K_{\epsilon}\ast \omega^N_s)  \omega^N_s- V^{N}\ast \left( (K_{\epsilon}\ast \omega^N_s)S_s^N\right)\right\|_{\L}\\
		%		\leq&\,\left\| (K_{\epsilon}\ast \omega^{\epsilon}_{s}) \omega^{\epsilon}_{s}- (K_{\epsilon}\ast \omega^N_s)  \omega^N_s\right\|_{\L}\\&+\,
		%		\left\|\int_{\R^d} V^N(x-y)\left( (K_{\epsilon}\ast \omega^N_s)(x)- (K_{\epsilon}\ast \omega^N_s)(y)\right)S_s^N(\d y)\right\|_{\L}\\
		=:&\, L_1+L_2.
	\end{aligned}
\end{equation*}

For the first term $L_1$, note that
\begin{equation*}
	(K_{\epsilon}\ast \omega^{\epsilon}_{s}) \omega^{\epsilon}_{s}- (K_{\epsilon}\ast \omega^N_s)  \omega^N_s=(K_{\epsilon}\ast (\omega^{\epsilon}_{s}-\omega^N_s))\omega^{\epsilon}_{s}+(K_{\epsilon}\ast \omega^N_s)(\omega^{\epsilon}_{s}-\omega^N_s);
\end{equation*}
using the bound $\sup_{\epsilon>0}\left\|\omega^{\epsilon}\right\|_{C([0,T],\L)}\leq C$ and \eqref{l1p}, it holds
\begin{equation*}
	\begin{aligned}
		&\, \left\|(K_{\epsilon}\ast (\omega^{\epsilon}_{s}-\omega^N_s))\omega^{\epsilon}_{s}\right\|_{\L}  \\ \leq&\,\left\|K_{\epsilon}\ast \left(\omega^{\epsilon}_{s}-\omega^N_s\right)\right\|_{\infty}\left\|\omega^{\epsilon}_{s}\right\|_{\L} \leq C\left\|K\ast\rho_{\eps}\ast \left(\omega^{\epsilon}_{s}-\omega^N_s\right)\right\|_{\infty} \\
		\leq & \,C_K\left\| \rho_{\eps}\ast \left(\omega^{\epsilon}_{s}-\omega^N_s\right)\right\|_{\L}\leq\,C_K\left\|  \omega^{\epsilon}_{s}-\omega^N_s\right\|_{\L}.
	\end{aligned}
\end{equation*}
Next, by \eqref{kg}, we have
\begin{equation*}
	\begin{aligned}
		\left\|(K_{\epsilon}\ast \omega^N_s)(\omega^{\epsilon}_{s}-\omega^N_s)\right\|_{\L}
		&\leq\left\|K_{\epsilon}\ast \omega_{s}^{N}\right\|_{\infty}\left\|\omega^{\epsilon}_{s}-\omega^N_s\right\|_{\L}\\
		& \leq\left\|K_{\epsilon}\right\|_{\infty}\left\|\omega^{\epsilon}_{s}-\omega^N_s\right\|_{\L}.
	\end{aligned}
\end{equation*}
Thus, for any small enough $\eps>0$,
\begin{equation}\label{kk}
	\begin{aligned}
		L_1
		\leq\, \big(C_K+\left\|K_{\epsilon}\right\|_{\infty}\big)\left\|\omega^{\epsilon}_{s}-\omega^N_s\right\|_{\L}
		\leq\, C\left\|K_{\epsilon}\right\|_{\infty}\left\|\omega^{\epsilon}_{s}-\omega^N_s\right\|_{\L},
	\end{aligned}
\end{equation}
where the constant $C$ is independent of $\epsilon$.
% $$\sup_{s \in[0, T]}\left\|\omega^{\epsilon}_{s}\right\|_{\L}\leq C.$$

For the second term $L_2$, we have
$$L_2=	\left\|\int_{\R^d} V^N(x-y)\left( (K_{\epsilon}\ast \omega^N_s)(x)- (K_{\epsilon}\ast \omega^N_s)(y)\right)S_s^N(\d y)\right\|_{\L}.$$
We observe that
$$
\begin{aligned}
	\left| (K_{\epsilon}\ast \omega^N_s)(x)- (K_{\epsilon}\ast \omega^N_s)(y)\right|
	&\leq\,\int_{\R^d}\left|K_{\epsilon}(x-z)-K_{\epsilon}(y-z)\right|\omega^N_s(z)\d z \\
	&\leq \left\|\nabla K_{\epsilon}\right\|_{\infty}\left|x-y\right|.
\end{aligned}
$$
Therefore, it yields
\begin{equation}\label{yx}
	\begin{aligned}
		L_2\leq	&\,\left\|\int_{\R^d} V^N(x-y)\left| (K_{\epsilon}\ast \omega^N_s)(x)- (K_{\epsilon}\ast \omega^N_s)(y)\right|S_s^N(\d y)\right\|_{\L}\\
		\leq &\,\left\|\nabla K_{\epsilon}\right\|_{\infty}\left\|\int_{\R^d} V^N(x-y)\left|x-y\right|S_s^N(\d y)\right\|_{\L}\\
		\leq&\,\frac{C\left\|\nabla K_{\epsilon}\right\|_{\infty}}{N^{\beta}}\left\|\omega_{s}^{N}\right\|_{\L},
	\end{aligned}
\end{equation}
where the last inequality is because that $V^N$ is supported on $\frac{\supp V}{N^{\beta}}$, where $\supp V$ denotes the support of $V$.

Substituting \eqref{kk} and \eqref{yx} into \eqref{integrandsecond} yields
\begin{equation}\label{secondomegaN}
	\begin{aligned}
		&\,\int_{0}^{t}\left\| \nabla\cdot e^{(t-s) A}\left( (K_\epsilon*\omega^{\epsilon}_{s}) \omega^{\epsilon}_{s}- V^{N}*\left( (K_{\epsilon}\ast \omega^N_s)S_s^N\right)\right)\right\|_{\L}\d s\\
		\leq&\,\frac{C\left\|\nabla K_{\epsilon}\right\|_{\infty}}{N^{\beta}}\int_{0}^{t} \frac{\left\|\omega_{s}^{N}\right\|_{\L}}{\sqrt{t-s}} \d s +C\left\|K_{\epsilon}\right\|_{\infty} \int_{0}^{t}\frac{\left\|\omega^{\epsilon}_{s}-\omega^N_s\right\|_{\L}}{\sqrt{t-s}} \d s.
	\end{aligned}
\end{equation}
Combining this estimate with \eqref{estimate} and using the simple inequality $ \big\| e^{tA} \big(\omega_0-\omega_{0}^{N}\big) \big\|_{\L}\leq \left\|\omega_0-\omega_{0}^{N} \right\|_{\L} $ , we deduce that
$$\aligned
\left\|\omega^{\epsilon}_t-\omega_{t}^{N}\right\|_{\L} &\leq \big\|\omega_0 - \omega^N_0 \big\|_{L^1\cap L^p} + \big\| Z^N_t \big\|_{L^1\cap L^p} \\
&\quad + \frac{C \|\nabla K_{\epsilon} \|_{\infty}}{N^{\beta}}\int_{0}^{t} \frac{\left\|\omega_{s}^{N}\right\|_{\L}}{\sqrt{t-s}} \d s +C \|K_{\epsilon} \|_{\infty} \int_{0}^{t}\frac{\left\|\omega^{\epsilon}_{s}-\omega^N_s\right\|_{\L}}{\sqrt{t-s}} \d s.
\endaligned $$
By Lemma \ref{Gronwall},
\begin{equation}\label{omegamittag}
	\begin{aligned}
		&\,\sup_{t \in[0, T]}\left\|\omega^{\epsilon}_t-\omega_{t}^{N}\right\|_{\L} \\
		\leq&\,\bigg( \left\|\omega_0-\omega_{0}^{N} \right\|_{\L} +\frac{C\left\|\nabla K_{\epsilon} \right\|_{\infty}}{N^{\beta}}\sup_{t \in[0, T]} \int_{0}^{t}\frac{\left\|\omega^N_s\right\|_{\L}}{\sqrt{t-s}} \d s+\sup_{t \in[0, T]}\left\|Z^N_{t}\right\|_{\L}	\bigg)\\
       & \times E_{\frac{1}{2}}\big(C_T\left\|K_{\epsilon}\right\|_{\infty}\big),\\
	\end{aligned}
\end{equation}
where, by \eqref{estimationmittag}, one has
\begin{equation}\label{estimationKmittag}
	E_{\frac{1}{2}}\left(C_T\left\|K_{\epsilon}\right\|_{\infty}\right)\leq 2\mathrm{exp}\big(C\left\|K_{\epsilon}\right\|_{\infty}^2\big).
\end{equation}
%In the second term of the second line of \eqref{omegamittag},
Since  $m>2$ and $m^{\prime}:=\frac{m}{m-1}\in(1,2)$, we have
\begin{equation*}%\label{estimationsecond}
	\begin{aligned}
		&\,	\bigg\|\sup_{t \in[0, T]}\int_{0}^{t}\frac{1}{\sqrt{t-s}}\left\|\omega^N_s\right\|_{\L}\d s \bigg\|_{L^m(\Omega)}\\
		\leq&  \left\|\sup_{t \in[0, T]}\bigg(\int_{0}^{t}\left\|\omega^N_s\right\|^m_{\L}\d s\bigg)^{\frac{1}{m}}\bigg(\int_{0}^{t} \frac{1}{(t-s)^{\frac{m'}2}}\, \d s\bigg)^{\frac{1}{m^{\prime}}}\right\|_{L^m(\Omega)}\\
		\leq&\,  C_{T,m} \bigg(\E\int_{0}^{T}\left\|\omega^N_s\right\|^m_{\L}\d s\bigg)^{\frac{1}{m}}
		\leq\, C_{T,m} \sup_{t \in[0, T]}\left\|\omega_{t}^{N}\right\|_{L^m(\Omega,\L)}.
	\end{aligned}
\end{equation*}
By Proposition \ref{estimategtN}, for any small enough $\epsilon>0$, we already have
\begin{equation*}
	\sup_{t \in[0, T]}\left\|\omega_{t}^{N}\right\|_{L^m(\Omega,L^p)} \leq C\left\|K_{\eps}\right\|_{\infty}^{\frac{d}{m}}\exp \big(C\left\|K_{\epsilon}\right\|_{\infty}^2 \big) \quad \text{and} \quad\sup_{t \in[0, T]}\left\|\omega_{t}^{N}\right\|_{L^m(\Omega,L^1)}=1,
\end{equation*}
then \begin{equation}\label{estimationsecond}
	\begin{aligned}
		\bigg\| \sup_{t \in[0, T]}\int_{0}^{t}\frac{1}{\sqrt{t-s}}\left\|\omega^N_s \right\|_{\L} \d s\bigg\|_{L^m(\Omega)}\leq\, C_T\left\|K_{\eps}\right\|_{\infty}^{\frac{d}{m}} \exp\big( C\left\|K_{\epsilon}\right\|_{\infty}^2 \big).
	\end{aligned}
\end{equation}
For the first term in the second line of \eqref{omegamittag}, by Hypotheses \ref{hypothesis-initial},
\begin{equation}%\label{initialomegaN}
	\left\|\omega_0-\omega_{0}^{N}\right\|_{L^m(\Omega,\L)}= \zeta_N\rightarrow 0\quad\text{as}\quad N\rightarrow\infty.
\end{equation}
For the stochastic convolution term in \eqref{omegamittag}, Theorem \ref{stochasticconvolution} yields that
\begin{equation}\label{estimationstocon}
	\bigg\| \sup_{t \in[0, T]}\left\| Z^N_t\right\|_{\L (\R^{d} )} \bigg\|_{L^m(\Omega)} \leq C\left\| K_{\epsilon} \right\|_{\infty}^{d+1} N^{-\iota},
	%CN^{\beta\left(1+3\delta\right)-\frac{1}{4m}}.
\end{equation}
where $\iota$ relies on the choice of $\beta$ in the potential function $V^N$. Recall the estimates in Section \ref{sectionkernel}:  $\left\| K_{\epsilon} \right\|_{\infty} \leq C\epsilon^{-d/p^{\prime}}$ and $\left\|\nabla K_{\epsilon}\right\|_{\infty}\leq C\epsilon^{-1-d/p^{\prime}}$ for any small enough $\epsilon>0$;  combining from \eqref{omegamittag} %\eqref{estimationinitial}, \eqref{estimationsecond} and
to \eqref{estimationstocon}, we deduce that
\begin{equation}\label{mainGronwall}
	\begin{aligned}
		&\,\bigg\| \sup_{t \in[0, T]}\left\|\omega^{\epsilon}_t-\omega_{t}^{N}\right\|_{\L} \bigg\|_{L^m(\Omega)}\\
		\leq&\,C\Bigg[ \zeta_N +\frac{\epsilon^{-1-\frac{d}{p^{\prime}}-\frac{d^2}{mp^{\prime}}}}{N^{\beta}} \exp\left(C^{\prime} \epsilon^{-\frac{2d}{p^{\prime}}} \right) +\frac{\epsilon^{-\frac{d(d+1)}{p^{\prime}}}}{N^{\iota}} \Bigg] \exp\left(C^{\prime}\epsilon^{-\frac{2d}{p^{\prime}}}\right).
	\end{aligned}
\end{equation}
In conclusion, we obtain the following result.

\begin{proposition}\label{main1}
	Assume Hypotheses \ref{hypothesis0}, \ref{hypothesis} and \ref{hypothesis-initial}; % We define a parameter $\gamma$ as
	%	\begin{equation*}
		%		\gamma:=\mathrm{min}\left\{\frac{\beta}{2},\iota\right\},
		%	\end{equation*}
	for some $\theta>0$ satisfying $C^{\prime}\theta<\mathrm{min}\big\{\frac{\beta}{2},\iota\big\}<1$ where $C^{\prime}$, $\beta$ and $\iota$ are the same as in \eqref{mainGronwall}, we take
	\begin{equation*}
		\epsilon= \epsilon(N)= \big[(\theta\log N)\wedge (-\theta\log \zeta_N)\big]^{-\frac{p^{\prime}}{2d}}.%=\left(\theta\log N\right)^{-\frac{p^{\prime}}{2d}}
	\end{equation*}
	%	and define $\gamma^{\prime}$ as
	%	\begin{equation*}
		%		\gamma^{\prime}:=\mathrm{min}\left\{\beta-2C^{\prime}\theta,\iota-C^{\prime}\theta\right\},
		%	\end{equation*}
	Then for any $T\in(0,T_{max})$ and small $\kappa>0$, the difference between the mollified empirical measure \eqref{mollifiedmeasure} and solution of \eqref{PDE-eps} can be estimated as
	\begin{equation}\label{secondpartlimit}
		\bigg\|\sup_{t \in[0, T]}\left\|\omega^{\epsilon}_t-\omega_{t}^{N}\right\|_{\L} \bigg\|_{L^m(\Omega)}\leq C\big(\zeta_N^{1-C^{\prime}\theta}\vee N^{-(\beta-2C^{\prime}\theta)+\kappa}\vee N^{-(\iota-C^{\prime}\theta)+\kappa} \big).
	\end{equation}
\end{proposition}

The proof is obvious since we notice that
$$\exp\left(C^{\prime}\epsilon(N)^{-\frac{2d}{p^{\prime}}}\right) =\exp\left((C^{\prime}\theta\log N)\wedge (-C^{\prime}\theta\log \zeta_N)\right)=N^{C^{\prime}\theta}\wedge \zeta_N^{-C^{\prime}\theta}, $$
and
\begin{equation*}%\label{mainGronwall}
	\begin{aligned}
		\bigg\| \sup_{t \in[0, T]}\left\|\omega^{\epsilon}_t-\omega_{t}^{N}\right\|_{\L} \bigg\|_{L^m(\Omega)}
		\leq\, C\Bigg[ \frac{\zeta_N}{\zeta_N^{C^{\prime}\theta}}+\frac{(\theta\log N)^{\frac{1}{2}\big(1+\frac{p^{\prime}}{d}+\frac{d}{m}\big)}}{N^{\beta-2C^{\prime}\theta}}
		+\frac{\left(\theta\log N\right)^{\frac{d+1}{2}}}{N^{\iota-C^{\prime}\theta}}\Bigg].
	\end{aligned}
\end{equation*}

\begin{proof}[Proof of Theorem \ref{thm-main}]
	If we also take $\epsilon=\epsilon(N)\sim \big[(\log N)\wedge (-\log \zeta_N)\big]^{-\frac{p^{\prime}}{2d}}$ in Proposition \ref{firstpart}, it holds that
	\begin{equation}\label{firstpartlimit}
		\lim\limits_{N\rightarrow\infty}\sup_{t \in[0, T]}\left\|\omega_t-\omega^{\epsilon}_{t}\right\|_{\L(\R^d)}=0.
	\end{equation}
	Therefore, substituting \eqref{secondpartlimit} and \eqref{firstpartlimit} into \eqref{split}, we finish the proof of Theorem \ref{thm-main}.
\end{proof}

%\section*{Appendices}

%\section{Appendices}
%\addcontentsline{toc}{section}{APPENDICES}

%%%%%%%%%%%%%%%%%%%%%%%%%%%%%%%%%%%%%%%%%%%%%%
%% Single Appendix:                         %%
%%%%%%%%%%%%%%%%%%%%%%%%%%%%%%%%%%%%%%%%%%%%%%
%\begin{appendix}
%\section*{???}%% if no title is needed, leave empty \section*{}.
%\end{appendix}
%%%%%%%%%%%%%%%%%%%%%%%%%%%%%%%%%%%%%%%%%%%%%%
%% Multiple Appendixes:                     %%
%%%%%%%%%%%%%%%%%%%%%%%%%%%%%%%%%%%%%%%%%%%%%%

\begin{appendix}

\section{Estimates on stochastic convolution}\label{appendixstochastic}

The purpose of this part is to prove Propositions \ref{supztp} and \ref{supzt1} concerning the stochastic convolution $Z_t^N$ defined in \eqref{mollifiedzt}. First, we present a useful lemma, which is a consequence of Garsia-Rodemich-Rumsey's Lemma \cite{GR70} for Banach spaces (see proof in \cite[Theorem A.1]{FK10}).

\begin{lemma}\label{sup}
	Let $E$ be a Banach space and $\left(Y^{N}\right)_{N \geq 1}$ a sequence of $E$-valued continuous processes on $[0, T]$. For $m \geq 1$ and $\eta>0$ satisfying $m \eta>1$, we assume that there exist constants $\rho>0$, $C>0$ and a sequence $\left(\chi_{N}\right)_{N \geq 1}$ of positive real numbers such that
	$$ \left\|Y_{t}^{N}-Y_{s}^{N}\right\|_{L^m(\Omega,E)}  \leq C|t-s|^{\eta} \chi_{N}^{\rho}, \quad \forall s, t \in[0, T],\, \forall N \geq 1.
	$$
	Then there exists a constant $C_{m, \eta, T}>0$, depending only on $m, \eta$ and $T$, such that for any $ N \geq 1$,
	$$
	\left(\E \bigg[\sup _{t \in[0, T]}\left\|Y_{t}^{N}-Y_{0}^{N}\right\|_{E}^{m} \bigg] \right)^{\frac{1}{m}} \leq C_{m, \eta, T} \, \chi_{N}^{\rho}.
	$$
\end{lemma}

To prove Propositions \ref{supztp} and \ref{supzt1}, we shall apply Lemma \ref{sup} to Banach spaces $E=L^p(\R^d)$ and $E=L^1(\R^d)$.

\subsection{Proof of Proposition \ref{supztp}}\label{sectiona3}

We consider the Banach space $E=L^p(\R^d)$ in Lemma \ref{sup}; recall the expression \eqref{mollifiedzt} of $Z_t^N$, for $p>2$ and for any $s,t\in[0,T]$,
\begin{equation}\label{m1m2p}
	\begin{aligned}
		&\,\big\|Z_t^N-Z^N_s \big\|_{L^m\left(\Omega,L^p\right)}\\
		\leq&\, \bigg\|\sum_{k} \int_{s}^{t} \nabla\cdot e^{(t-\tau) A}\left( V^{N} \ast \left(\sigma_k^N S_{\tau}^{N}\right)\right)  \d W_{\tau}^{k} \bigg\|_{L^m\left(\Omega,L^p\right)}\\
		&\,+\bigg\|\sum_{k} \int_{0}^{s} \nabla\cdot e^{(s-\tau) A}\left( e^{(t-s)A}\left(V^{N} \ast \left(\sigma_k^N S_{\tau}^{N}\right)\right)-V^{N} \ast \left(\sigma_k^N S_{\tau}^{N}\right)\right)  \d W_{\tau}^{k} \bigg\|_{L^m\left(\Omega,L^p\right)}\\
		=:&\,M_1+M_2.
	\end{aligned}
\end{equation}	
\noindent In the sequel, we will estimate $M_1$ and $M_2$ respectively.
%\begin{itemize}
%	\item[\textbf{Step 1.}] \textbf{The bound of $M_1$ in \eqref{m1m2p}}.
%\end{itemize}

\smallskip
\noindent\textbf{Step 1. The bound of $M_1$ in \eqref{m1m2p}}. \smallskip

It is easy to show that
\begin{equation}\label{M-1.1}
	e^{(t-\tau) A}\left( V^{N} \ast \left(\sigma_k^N S_{\tau}^{N}\right)\right)(\cdot) =\frac{1}{N}\sum_{i=1}^{N} \sigma_{k}^{N}(X_{\tau}^{N,i}) \big(e^{(t-\tau) A} V^{N}\big) \left(\cdot-X_{\tau}^{N,i}\right),
\end{equation}
thus,
$$
\begin{aligned}
	M_1=
	&\,\bigg\|\sum_{k} \int_{s}^{t} \frac{1}{N}\sum_{i=1}^{N}\nabla\cdot \Big[\sigma_{k}^{N}(X_{\tau}^{N,i})\big(e^{(t-\tau) A} V^{N}\big) \left(\cdot-X_{\tau}^{N,i}\right) \Big]\d W_{\tau}^{k} \bigg\|_{L^m\left(\Omega,L^p\right)}\\
	=&\, \bigg\|\sum_{k} \int_{s}^{t}\frac{1}{N} \sum_{i=1}^{N}\sigma_{k}^{N}(X_{\tau}^{N,i})\cdot  \big( \nabla e^{(t-\tau) A} V^{N}\big)\left(\cdot-X_{\tau}^{N,i}\right) \d W_{\tau}^{k} \bigg\|_{L^{m}(\Omega,L^{p})}.\\
\end{aligned}
$$
Then applying Sobolev embedding $H^{d(\frac{1}{2}-\frac{1}{p}),2}(\R^d)\subset L^p(\R^d)$, we have
{\small \begin{equation*}
		\begin{aligned}
			M_1^m
			\leq&\, C \E\bigg\|\sum_{k} \int_{s}^{t}\frac{1}{N} \sum_{i=1}^{N}\sigma_{k}^{N}(X_{\tau}^{N,i})\cdot \! \Big((\mathrm{I}-A)^{\frac{d}{2}(\frac{1}{2}-\frac{1}{p})}   \nabla e^{(t-\tau) A} V^{N}\Big)\! \big(\cdot-X_{\tau}^{N,i}\big) \d W_{\tau}^{k} \bigg\|_{L^2}^m\\
			\leq&\, C\E\Bigg[ \sum_{k}\int_{s}^{t} \bigg\|\frac{1}{N} \sum_{i=1}^{N} \sigma_{k}^{N}(X_{\tau}^{N,i}) \cdot \! \Big((\mathrm{I}-A)^{\frac{d}{2}(\frac{1}{2}-\frac{1}{p})}  \nabla e^{(t-\tau) A} V^{N}\Big)\! \big(\cdot-X_{\tau}^{N,i}\big) \bigg\|_{L^{2}}^{2} \d \tau \Bigg]^{\frac{m}{2}}.\\
		\end{aligned}
\end{equation*} }

\noindent Define the operator
\begin{equation}\label{operatorlt}
	\mathcal{L}_{t,\tau}:=(\mathrm{I}-A)^{\frac{d}{2}(\frac{1}{2}-\frac{1}{p})}  \nabla e^{(t-\tau) A}, \quad t>\tau\geq 0;
\end{equation}
by \eqref{defQN}, we have
\begin{equation*}
	\begin{aligned}
		&\,\sum_{k} \bigg\|\frac{1}{N} \sum_{i=1}^{N}\sigma_{k}^{N}(X_{\tau}^{N,i})\cdot\mathcal{L}_{t,\tau}
		V^{N} \big(\cdot-X_{\tau}^{N,i}\big) \bigg\|_{L^{2}}^{2}\\
		=&\,\sum_{k}\! \int_{\R^d}\! \frac{1}{N^2}\!\! \sum_{i,j=1}^{N} \! \big[\sigma_{k}^{N}(X_{\tau}^{N,i}) \cdot\mathcal{L}_{t,\tau}
		V^{N}\! \big(x-X_{\tau}^{N,i}\big)\big] \big[\sigma_{k}^{N}(X_{\tau}^{N,j})\cdot\mathcal{L}_{t,\tau}
		V^{N}\! \big(x-X_{\tau}^{N,j}\big) \big]\d x\\
		\leq&\, \frac{1}{N^2}\sum_{i,j=1}^{N}\left|Q_N(X_{\tau}^{N,i}-X_{\tau}^{N,j})\right| \int_{\R^d}\left| \mathcal{L}_{t,\tau} V^{N} \big(x-X_{\tau}^{N,i}\big)\right|\left|\mathcal{L}_{t,\tau}
		V^{N} \big(x-X_{\tau}^{N,j}\big) \right|\d x,
	\end{aligned}
\end{equation*}
which, by Cauchy-Schwarz inequality, is dominated by
$$
\frac{1}{N^2}\sum_{i,j=1}^{N}\left|Q_N(X_{\tau}^{N,i}-X_{\tau}^{N,j})\right|\left\|\mathcal{L}_{t,\tau}
V^{N}\right\|^{2}_{L^2}.
$$
As a result,
\begin{equation}\label{ELL}
	\begin{aligned}
		\xi_N:=	&\,\sum_{k}\int_{s}^{t} \bigg\|\frac{1}{N} \sum_{i=1}^{N}\sigma_{k}^{N}(X_{\tau}^{N,i})\cdot\mathcal{L}_{t,\tau}
		V^{N}\left(\cdot-X_{\tau}^{N,i}\right) \bigg\|_{L^{2}}^{2} \d \tau\\
		%\leq&\,\frac{1}{N^2}\sum_{i,j=1}^{N}\int_{s}^{t}\left|Q_N(X_{\tau}^{N,i}-X_{\tau}^{N,j})\right|\left(\int_{\R^d}\left|\mathcal{L}_{t,\tau}
		%	V^{N}\left(x-X_{\tau}^{N,i}\right)\right|\left|\mathcal{L}_{t,\tau}
		%	V^{N}\left(x-X_{\tau}^{N,j}\right)\right|\d x\right)\d \tau
		\leq&\,  \frac{1}{N^2}\sum_{i,j=1}^{N}\int_{s}^{t}\left|Q_N(X_{\tau}^{N,i}-X_{\tau}^{N,j})\right|\left\|\mathcal{L}_{t,\tau}
		V^{N}\right\|^{2}_{L^2}\d \tau\\
		\leq&\,\frac{1}{N^2}\sum_{i,j=1}^{N}\left(\int_{s}^{t}\left|Q_N(X_{\tau}^{N,i}-X_{\tau}^{N,j})\right|^m\d \tau\right)^{\frac{1}{m}}\left(\int_{s}^{t}\left\|\mathcal{L}_{t,\tau}
		V^{N}\right\|^{2m^{\prime}}_{L^2}\d \tau\right)^{\frac{1}{m^{\prime}}},
	\end{aligned}
\end{equation}
where $m^{\prime}$ denotes the conjugate number of $m$. Then, Jensen's inequality yields
\begin{equation}\label{EQL}
	\begin{aligned}
		M_1^m\leq&\,	C\E\left[\xi_N\right]^{\frac{m}{2}}\\ \leq&\,C\E\Bigg[\frac{1}{N^2}\sum_{i,j=1}^{N}\left(\int_{s}^{t}\left|Q_N(X_{\tau}^{N,i}-X_{\tau}^{N,j})\right|^m\d \tau\right)^{\frac{1}{m}} \Bigg]^{\frac{m}{2}}\left(\int_{s}^{t}\left\|\mathcal{L}_{t,\tau}
		V^{N}\right\|^{2m^{\prime}}_{L^2}\d \tau\right)^{\frac{m}{2m^{\prime}}}\\
		\leq&\,\frac{C}{N^2}\sum_{i,j=1}^{N}\E\left[\int_{s}^{t}\left|Q_N(X_{\tau}^{N,i}-X_{\tau}^{N,j})\right|^m\d \tau\right]^{\frac{1}{2}}\left(\int_{s}^{t}\left\|\mathcal{L}_{t,\tau}
		V^{N}\right\|^{2m^{\prime}}_{L^2}\d \tau\right)^{\frac{m}{2m^{\prime}}}.\\	
	\end{aligned}
\end{equation}

For $\delta<1$ we have
$$\aligned
\left\|\mathcal{L}_{t,\tau} V^{N}\right\|_{L^2} =&\,\left\|(\mathrm{I}-A)^{\frac{d}{2}(\frac{1}{2}-\frac{1}{p})} \nabla e^{(t-\tau) A} V^{N}\right\|_{L^{2}}\\
=&\,\left\|(\mathrm{I}-A)^{-\frac{\delta}{2}} \nabla e^{(t-\tau) A}(\mathrm{I}-A)^{\frac{\delta}{2}+\frac{d}{2}(\frac{1}{2}-\frac{1}{p})} V^{N}\right\|_{L^{2}};
\endaligned $$
then by \eqref{lpestimate},
\begin{equation}\label{estimatelst}
	\begin{aligned}
		\left\|\mathcal{L}_{t,\tau} V^{N}\right\|_{L^2} \leq \frac{C_{\delta,\nu,T}}{(t-\tau)^{\frac{1-\delta}{2}}} \left\|V^{N}\right\|_{\delta+d(\frac{1}{2}-\frac{1}{p}), 2} \leq  \frac{C_{\delta,\nu,T}}{(t-\tau)^{\frac{1-\delta}{2}}} N^{\beta \big(d+\delta-\frac{d}{p} \big)}.
	\end{aligned}
\end{equation}
By Hypotheses \ref{hypothesis}, we have $\frac{2}{m}<\delta<1$ which implies $m^{\prime}(1-\delta)<1$; consequently,
\begin{equation}\label{lst}
	\begin{aligned}
		\left(\int_{s}^{t}\left\|\mathcal{L}_{t,\tau} V^{N}\right\|^{2m^{\prime}}_{L^2} \d \tau\right)^{\frac{1}{2m^{\prime}}}
		\leq&\, C_{\delta,\nu,T}N^{\beta \big(d+\delta-\frac{d}{p} \big)} \left(\int_{s}^{t} \frac{1}{(t-\tau)^{m^{\prime}(1-\delta)}} \d \tau\right)^{\frac{1}{2m^{\prime}}}\\
		\leq&\, C_{\delta,\nu,T} N^{\beta\big(d+\delta-\frac{d}{p}\big)} \left(t-s\right)^{\frac{1}{2}\left(\delta-\frac{1}{m}\right)}.
	\end{aligned}
\end{equation}
Applying Lemma \ref{convratelemma}, we obtain
\begin{equation}\label{NEQ}
	\frac{1}{N^2}\sum_{i,j=1}^{N}\left(\E\left[\int_{s}^{t}\left|Q_N(X_{\tau}^{N,i}-X_{\tau}^{N,j})\right|^m\d \tau\right]\right)^{\frac{1}{2}}\leq C\left(t-s\right)^{\frac{1}{2}}\left\|K_{\epsilon}\right\|_{\infty}^{d}N^{-\frac{1}{4}}.
\end{equation}
Substituting \eqref{lst} and \eqref{NEQ} into \eqref{EQL} yields that
\begin{equation}\label{step1p}
	M_1\leq  C_{\delta}\left\|K_{\epsilon}\right\|_{\infty}^{\frac{d}{m}}N^{\beta\big(d+\delta-\frac{d}{p}\big) -\frac{1}{4m}} (t-s )^{\frac{\delta}{2}}.
\end{equation}

\smallskip

\noindent\textbf{Step 2. The bound of $M_2$ in \eqref{m1m2p}}. \smallskip

Similarly to \eqref{M-1.1}, and using Sobolev embedding theorem, we have
$$\aligned
&\, \bigg\|\sum_{k} \int_{0}^{s} \nabla\cdot e^{(s-\tau) A}\left( e^{(t-s)A}\left(V^{N} \ast \left(\sigma_k^N S_{\tau}^{N}\right)\right)-V^{N} \ast \left(\sigma_k^N S_{\tau}^{N}\right)\right)  \d W_{\tau}^{k} \bigg\|_{L^p} \\
=&\, \bigg\|\sum_{k} \int_{0}^{s}\frac{1}{N} \sum_{i=1}^{N}\sigma_{k}^{N}(X_{\tau}^{N,i})\cdot   \Big[\nabla e^{(s-\tau) A} \big(e^{(t-s)A}-\mathrm{I}\big)V^{N} \Big] \big(\cdot-X_{\tau}^{N,i} \big)\, \d W_{\tau}^{k} \bigg\|_{L^p} \\
\leq &\, C \bigg\|\sum_{k} \int_{0}^{s}\frac{1}{N} \sum_{i=1}^{N}\sigma_{k}^{N}(X_{\tau}^{N,i})\cdot  \big[\mathcal{L}_{s,\tau} \big(e^{(t-s)A}-\mathrm{I}\big)V^{N} \big]\big(\cdot-X_{\tau}^{N,i} \big)\, \d W_{\tau}^{k} \bigg\|_{L^2},
\endaligned$$
where the operator $\mathcal{L}_{s,\tau}$ is defined in \eqref{operatorlt}. Hence,
\begin{equation*}
	\begin{aligned}
		M_2^m\leq&\, C_m \E \bigg\|\sum_{k} \int_{0}^{s}\frac{1}{N} \sum_{i=1}^{N}\sigma_{k}^{N}(X_{\tau}^{N,i})\cdot  \mathcal{L}_{s,\tau} \big(e^{(t-s)A}-\mathrm{I}\big)V^{N}\left(\cdot-X_{\tau}^{N,i}\right) \d W_{\tau}^{k} \bigg\|_{L^2}^m \\
		\leq&\, C'_m \E\Bigg[ \sum_{k}\int_{0}^{s} \bigg\|\frac{1}{N} \sum_{i=1}^{N}\sigma_{k}^{N}(X_{\tau}^{N,i})\cdot\mathcal{L}_{s,\tau} \big(e^{(t-s)A}-\mathrm{I}\big)V^{N}\left(\cdot-X_{\tau}^{N,i}\right) \bigg\|_{L^{2}}^{2} \d \tau \Bigg]^{\frac{m}{2}}.
	\end{aligned}
\end{equation*}
Following the ideas in the treatments of \eqref{ELL} and \eqref{EQL}, we replace $V^N$ by $\big(e^{(t-s)A}-\mathrm{I}\big)V^{N}$, integrate from $0$ to $s$ and obtain
\begin{equation*}
	\begin{aligned}
		&\,\E\Bigg[\sum_{k}\int_{0}^{s} \bigg\| \frac{1}{N} \sum_{i=1}^{N}\sigma_{k}^{N}(X_{\tau}^{N,i})\cdot\mathcal{L}_{s,\tau}
		\big(e^{(t-s)A}-\mathrm{I}\big)V^{N}\left(\cdot-X_{\tau}^{N,i}\right) \bigg\|_{L^{2}}^{2} \d \tau \Bigg]^{\frac{m}{2}}\\
		\leq&\,\E\Bigg[\frac{1}{N^2}\! \sum_{i,j=1}^{N}\! \left(\int_{0}^{s}\! \left|Q_N(X_{\tau}^{N,i}-X_{\tau}^{N,j})\right|^m\d \tau\! \right)^{\!\! \frac{1}{m}} \!\! \Bigg]^{\!\frac{m}{2}} \! \bigg[\int_{0}^{s}\! \big\|\mathcal{L}_{s,\tau} \big(e^{(t-s)A}-\mathrm{I}\big)V^{N} \big\|^{2m^{\prime}}_{L^2}\! \d \tau \bigg]^{\! \frac{m}{2m^{\prime}}}.
	\end{aligned}
\end{equation*}
Similarly to the estimate on $M_1$, it yields
\begin{equation*}
	\begin{aligned}
		M_2
		\leq&\,C\Bigg\{\frac{1}{N^2}\sum_{i,j=1}^{N} \E\left[\int_{0}^{s}\left|Q_N(X_{\tau}^{N,i}-X_{\tau}^{N,j})\right|^m\d \tau\right]^{\frac{1}{2}}\Bigg\}^{\frac{1}{m}} \\
		&\, \times \bigg[\int_{0}^{s} \big\|\mathcal{L}_{s,\tau} \big(e^{(t-s)A}-\mathrm{I}\big)V^{N} \big\|^{2m^{\prime}}_{L^2} \d \tau \bigg]^{\frac{1}{2m^{\prime}}}\\
		\leq&\, C\left\|K_{\epsilon}\right\|_{\infty}^{\frac{d}{m}}N^{-\frac{1}{4m}} \bigg[\int_{0}^{s} \big\|\mathcal{L}_{s,\tau} \big(e^{(t-s)A}-\mathrm{I}\big)V^{N} \big\|^{2m^{\prime}}_{L^2}\d \tau \bigg]^{\frac{1}{2m^{\prime}}}.
	\end{aligned}
\end{equation*}

To get the same term $\left(t-s\right)^{\frac{\delta}{2}}$ as in the bound of $M_1$, we estimate $M_2$ in two ways. For $f\in H^1(\R^d)$, it holds
\begin{equation}\label{etsestimate1}
	\big\|\big(e^{(t-s)A}-\mathrm{I}\big)f \big\|_{L^2}\leq\left\|\nabla f\right\|_{L^2}\sqrt{t-s}.
\end{equation}
Recalling the definition of $\mathcal{L}_{s,\tau}$ in \eqref{operatorlt}, we have
\begin{equation*}
	\begin{aligned}
		\big\|\mathcal{L}_{s,\tau} \big(e^{(t-s)A}-\mathrm{I}\big)V^{N}\big\|_{L^2} =&\,\left\|(\mathrm{I}-A)^{\frac{d}{2}(\frac{1}{2}-\frac{1}{p})} \nabla e^{(s-\tau) A} 	\big(e^{(t-s)A}-\mathrm{I}\big)V^{N}\right\|_{L^{2}}\\
		=&\,\left\|(\mathrm{I}-A)^{-\frac{\delta}{2}} \nabla e^{(s-\tau) A}	\big(e^{(t-s)A}-\mathrm{I}\big)(\mathrm{I}-A)^{\frac{\delta}{2}+\frac{d}{2}(\frac{1}{2}-\frac{1}{p})} V^{N}\right\|_{L^{2}}.\\
	\end{aligned}
\end{equation*}
Taking $f=(\mathrm{I}-A)^{\frac{\delta}{2}+\frac{d}{2}(\frac{1}{2}-\frac{1}{p})} V^{N}$, by \eqref{etsestimate1},
\begin{equation*}
	\begin{aligned}
		&\, \big\|\mathcal{L}_{s,\tau} \big(e^{(t-s)A}-\mathrm{I}\big)V^{N} \big\|_{L^2} \\
		\leq&\,\big\|(\mathrm{I}-A)^{-\frac{\delta}{2}} \nabla e^{(s-\tau) A} \big\|_{L^{2}\rightarrow L^{2}} \left\| \nabla(\mathrm{I}-A)^{\frac{\delta}{2} +\frac{d}{2}(\frac{1}{2}-\frac{1}{p})} V^{N}\right\|_{L^{2}}\sqrt{t-s}\\
		\leq&\, \frac{C_{\delta,\nu,T}\sqrt{t-s}}{(s-\tau)^{\frac{1-\delta}{2}}}\left\|V^{N}\right\|_{1+\delta+d(\frac{1}{2}-\frac{1}{p}), 2} \leq \frac{C_{\delta,\nu,T}\sqrt{t-s}}{(s-\tau)^{\frac{1-\delta}{2}}} N^{\beta\big(d+1+\delta-\frac{d}{p}\big)}.
	\end{aligned}
\end{equation*}
Then,
\begin{equation}\label{M21}
	\begin{aligned}
		M_2\leq  C_{\delta,\nu,T}\left\|K_{\epsilon}\right\|_{\infty}^{\frac{d}{m}} N^{\beta\big(d+1+\delta-\frac{d}{p}\big) -\frac{1}{4m}}\sqrt{t-s}.%\left(t-s\right)^{\frac{1}{2}\left(\frac{1}{m^{\prime}}+\delta-1\right)}
	\end{aligned}
\end{equation}
If we use the estimate
\begin{equation}\label{etsestimate2}
	\big\|\big(e^{(t-s)A}-\mathrm{I}\big)f \big\|_{L^2} \leq 2 \| f \|_{L^2}
\end{equation}
rather than \eqref{etsestimate1}, still taking $f=(\mathrm{I}-A)^{\frac{\delta}{2}+\frac{d}{2}(\frac{1}{2}-\frac{1}{p})} V^{N}$, then
\eqref{estimatelst} implies that
\begin{equation*}
	\begin{aligned}
		\big\|\mathcal{L}_{s,\tau} 	\big(e^{(t-s)A}-\mathrm{I}\big)V^{N} \big\|_{L^2} \leq&\,2\left\|(\mathrm{I}-A)^{-\frac{\delta}{2}} \nabla e^{(s-\tau) A}	\right\|_{L^{2}\rightarrow L^{2}}\left\|(\mathrm{I}-A)^{\frac{\delta}{2}+\frac{d}{2}(\frac{1}{2}-\frac{1}{p})} V^{N}\right\|_{L^{2}} \\ \leq&\, \frac{C_{\delta,\nu,T}}{(s-\tau)^{\frac{1-\delta}{2}}} N^{\beta\big(d+\delta-\frac{d}{p}\big)}.
	\end{aligned}
\end{equation*}
Thus,
\begin{equation}\label{M22}
	M_2\leq  C_{\delta,\nu,T}\left\|K_{\epsilon}\right\|_{\infty}^{\frac{d}{m}} N^{\beta\big(d+\delta-\frac{d}{p}\big)-\frac{1}{4m}}.
\end{equation}
Combining \eqref{M21} and \eqref{M22}, we deduce by interpolation $M_2=M_2^{\delta}M_2^{1-\delta}$ that
\begin{equation}\label{step2p}
	M_2\leq C_{\delta}\left\|K_{\epsilon}\right\|_{\infty}^{\frac{d}{m}} N^{\beta\big(d +2\delta-\frac{d}{p}\big) -\frac{1}{4m}} (t-s )^{\frac{\delta}{2}}.
\end{equation}

\smallskip

\noindent\textbf{Step 3. The bound of \eqref{m1m2p}}. \smallskip

Substituting \eqref{step1p} and \eqref{step2p} into  \eqref{m1m2p}, we obtain
\begin{equation*}
	\begin{aligned}
		\left\|Z_t^N-Z^N_s\right\|_{L^m\left(\Omega,L^p\right)}
		\leq &\,C_{\delta}\left\|K_{\epsilon}\right\|_{\infty}^{\frac{d}{m}} \Big( N^{\beta \big(d+\delta-\frac{d}{p} \big) -\frac{1}{4m}} +N^{\beta\big(d +2\delta-\frac{d}{p}\big) -\frac{1}{4m}}\Big) \left(t-s\right)^{\frac{\delta}{2}}\\
		\leq &\,C_{\delta} \left\|K_{\epsilon}\right\|_{\infty}^{\frac{d}{m}} N^{\beta\big(d +2\delta-\frac{d}{p} \big) -\frac{1}{4m}}\left(t-s\right)^{\frac{\delta}{2}}.
	\end{aligned}
\end{equation*}
Applying Lemma \ref{sup} with $\eta=\frac{\delta}{2}$ which satisfies $m\eta>1$, we deduce that
\begin{equation*}%\label{zp}
	\bigg\|\sup_{t \in[0, T]}\big\| Z^N_t \big\|_{L^{p}\left(\R^{d}\right)} \bigg\|_{L^m(\Omega)}\leq C_{m,\delta,T}\left\|K_{\epsilon}\right\|_{\infty}^{\frac{d}{m}} N^{\beta\big(d +2\delta-\frac{d}{p}\big)-\frac{1}{4m}},
\end{equation*}
where $\frac{2}{m}<\delta<1$ and  $\beta$ satisfies $$0<\beta\leq\frac{1}{4m(d+2)}<\frac{1}{4m(d+2\delta-\frac{d}{p})}.$$
%then
%\begin{equation*}%\label{zp}
%	\left\|\sup_{t \in[0, T]}\left\| Z^N_t\right\|_{L^{p}\left(\R^{d}\right)}\right\|_{L^m(\Omega)}\leq C_{m,T}\left\|K_{\epsilon}\right\|_{\infty}^{\frac{d}{m}}N^{\beta\left(d+2\right)-\frac{1}{4m}}.
%\end{equation*}
%We assume $\delta>1$ is because that we require $m\eta>1$ is Lemma \ref{sup}.

\subsection{Proof of Proposition \ref{supzt1}}\label{sectiona4}

Here we consider the Banach space $E=L^1(\R^d)$ in Lemma \ref{sup}; similarly to the proof of Proposition \ref{supztp}, for any $s,t\in[0,T]$,
\begin{equation}\label{m1m21}
	\begin{aligned}
		&\,\big\|Z_t^N-Z^N_s \big\|_{L^m\left(\Omega,L^1\right)}\\ \leq&\, \bigg\|\sum_{k} \int_{s}^{t} \nabla\cdot e^{(t-\tau) A}\left( V^{N} \ast \left(\sigma_k^N S_{\tau}^{N}\right)\right)  \d W_{\tau}^{k} \bigg\|_{L^m\left(\Omega,L^1\right)}\\
		&\,+ \bigg\|\sum_{k} \int_{0}^{s} \nabla\cdot e^{(s-\tau) A}\left( e^{(t-s)A}\left(V^{N} \ast \left(\sigma_k^N S_{\tau}^{N}\right)\right)-V^{N} \ast \left(\sigma_k^N S_{\tau}^{N}\right)\right)  \d W_{\tau}^{k} \bigg\|_{L^m\left(\Omega,L^1\right)}\\
		=:&\,\tilde{M}_1 + \tilde M_2.
	\end{aligned}
\end{equation}		

\smallskip
\noindent\textbf{Step 1. The bound of $\tilde M_1$ in \eqref{m1m21}}. \smallskip

Since $L^1(\R^d)$ is not a UMD Banach space, we will follow the idea in \cite{ORT20b} and multiply the functions by the weight $w(x):= 1+|x|^{\frac{d+1}{2}},\, x\in \R^d$:
$$\aligned
\tilde M_1^m\leq&\, \E \bigg\|\sum_{k} \int_{s}^{t}\frac{1}{N} \sum_{i=1}^{N}\sigma_{k}^{N}(X_{\tau}^{N,i})\cdot   \big(\nabla e^{(t-\tau) A} V^{N}\big)\big(\cdot-X_{\tau}^{N,i} \big) \d W_{\tau}^{k} \bigg\|_{L^1}^m\\
\leq&\, \E\Bigg\|\frac{1}{w(\cdot)} \sum_{k} \int_{s}^{t}\frac{1}{N} \sum_{i=1}^{N}\sigma_{k}^{N}(X_{\tau}^{N,i})\cdot  w(\cdot) \big(\nabla e^{(t-\tau) A} V^{N}\big)\big(\cdot-X_{\tau}^{N,i}\big) \d W_{\tau}^{k} \Bigg\|_{L^1}^m .
\endaligned $$
Applying Cauchy-Schwarz inequality, we have
\begin{equation}\label{m1L11}
	\begin{aligned}
		\tilde M_1^m \leq&\, C_d\, \E \bigg\|\sum_{k} \int_{s}^{t}\frac{1}{N} \sum_{i=1}^{N}\sigma_{k}^{N}(X_{\tau}^{N,i})\cdot  w(\cdot) \big(\nabla e^{(t-\tau) A} V^{N}\big)\left(\cdot-X_{\tau}^{N,i}\right) \d W_{\tau}^{k} \bigg\|_{L^2}^m \\
		\leq&\,  C_d\, \E\Bigg[\! \sum_{k}\!\int_{s}^{t}\! \bigg\|\frac{1}{N} \sum_{i=1}^{N} \sigma_{k}^{N}(X_{\tau}^{N,i}) \cdot w(\cdot)  \big(\nabla e^{(t-\tau) A} V^{N}\big)\big(\cdot-X_{\tau}^{N,i}\big) \bigg\|_{L^{2}}^{2} \d \tau\Bigg]^{\frac{m}{2}}.\\
	\end{aligned}
\end{equation}
For notational simplicity, we denote $f_{\tau}(x, X_{\tau}^{N,i} ):= w(x) \nabla e^{(t-\tau) A} V^{N}\big(x-X_{\tau}^{N,i}\big)$, $x\in\R^d$; then,
\begin{equation*}%\label{ELL1}
	\begin{aligned}
		\tilde{\xi}_N:=&\,\sum_{k}\int_{s}^{t} \bigg\|\frac{1}{N} \sum_{i=1}^{N}\sigma_{k}^{N}(X_{\tau}^{N,i})\cdot f_{\tau}(\cdot, X_{\tau}^{N,i} ) \bigg\|_{L^{2}}^{2} \d \tau\\
		=&\,\sum_{k}\int_{s}^{t}\int_{\R^d}\frac{1}{N^2}\sum_{i,j=1}^{N}\left(\sigma_{k}^{N}(X_{\tau}^{N,i})\cdot f_{\tau}(x, X_{\tau}^{N,i} )\right)\left(\sigma_{k}^{N}(X_{\tau}^{N,j})\cdot f_{\tau}(x, X_{\tau}^{N,j} )\right)\d x\d \tau\\
		\leq&\,\frac{1}{N^2}\sum_{i,j=1}^{N}\int_{s}^{t}\left|Q_N(X_{\tau}^{N,i}-X_{\tau}^{N,j})\right|\left(\int_{\R^d}\left|f_{\tau}(x, X_{\tau}^{N,i} )\right|\left|f_{\tau}(x, X_{\tau}^{N,j} )\right|\d x\right)\d \tau ,
	\end{aligned}
\end{equation*}
where in the last step we have used the identity \eqref{defQN}. H\"older's inequality leads to
$$\aligned \tilde{\xi}_N \leq \frac{1}{N^2}\sum_{i,j=1}^{N}& \left[\int_{s}^{t}\left|Q_N(X_{\tau}^{N,i}-X_{\tau}^{N,j})\right|^m\d \tau\right]^{\frac{1}{m}} \\
& \times \bigg[ \int_{s}^{t}\left(\int_{\R^d}\left|f_{\tau}(x, X_{\tau}^{N,i} )\right|\left|f_{\tau}(x, X_{\tau}^{N,j} )\right|\d x\right)^{m^{\prime}}\d \tau \bigg]^{\frac{1}{m^{\prime}}}.
\endaligned $$
Applying Jensen's inequality, we obtain
$$\begin{aligned}
	\tilde{\xi}_N^{\, m/2} \leq \frac{1}{N^2}\sum_{i,j=1}^{N}
&\left[\int_{s}^{t}\left|Q_N(X_{\tau}^{N,i}-X_{\tau}^{N,j})\right|^m\d \tau\right]^{\frac{1}{2}} \\
&\times\bigg[ \int_{s}^{t}\left(\int_{\R^d}\left|f_{\tau}(x, X_{\tau}^{N,i} )\right| \left|f_{\tau}(x, X_{\tau}^{N,j}) \right| \d x\right)^{m^{\prime}}\d \tau \bigg]^{\frac{m}{2m^{\prime}}} .
\end{aligned}$$
Then, by Cauchy-Schwarz inequality,
\begin{equation*}%\label{EQL1}
	\begin{aligned}
		\E\Big[\tilde{\xi}_N^{\, m/2} \Big]	%&\,\E\left[\sum_{k}\int_{s}^{t}\left\|\frac{1}{N} \sum_{i=1}^{N}\sigma_{k}^{N}(X_{\tau}^{N,i})\cdot f_{\tau}(\cdot, X_{\tau}^{N,i} )\right\|_{L^{2}}^{2} \d \tau\right]^{\frac{m}{2}}\\
		\leq&\,\frac{1}{N^2}\sum_{i,j=1}^{N}\left[\E \int_{s}^{t}\left|Q_N(X_{\tau}^{N,i}-X_{\tau}^{N,j})\right|^{m}\d \tau \right]^{\frac{1}{2}}\\ &\qquad\qquad\times \left\{\E \bigg[\int_{s}^{t}\left(\int_{\R^d}\left|f_{\tau}(x, X_{\tau}^{N,i} )\right|\left|f_{\tau}(x, X_{\tau}^{N,j} )\right|\d x\right)^{m^{\prime}}\d \tau \bigg]^{\frac{m}{m^{\prime}}} \right\}^{\frac{1}{2}}.\\
	\end{aligned}
\end{equation*}
Lemma \ref{convratelemma} implies that
{\small \begin{equation}\label{EQL12}
		\begin{aligned}
			\E\Big[\tilde{\xi}_N^{\, m/2} \Big]	\leq \frac{C \|K_{\epsilon} \|_{\infty}^{d} (t-s )^{\frac{1}{2}}}{N^{\frac{1}{4}+2}}\! \sum_{i,j=1}^{N}\! \left\{ \E\bigg[ \int_{s}^{t}\!\! \bigg(\int_{\R^d}\! \big|f_{\tau}(x, X_{\tau}^{N,i} )\big|\, \big|f_{\tau}(x, X_{\tau}^{N,j} )\big|\d x \bigg)^{\! m^{\prime}} \! \d \tau \bigg]^{\! \frac{m}{m^{\prime}}}\! \right\}^{\! \frac{1}{2}}.\\
		\end{aligned}
\end{equation} }

\noindent To bound \eqref{EQL12}, we observe
\begin{equation}\label{i1}
	\begin{aligned}
		&\,\E\left[\int_{s}^{t}\left(\int_{\R^d}\left|f_{\tau}(x, X_{\tau}^{N,i} )\right|\left|f_{\tau}(x, X_{\tau}^{N,j} )\right|\d x\right)^{m^{\prime}}\d \tau\right]^{\frac{m}{m^{\prime}}}\\
		\leq&\,C\E\left\|\int_{\R^d}\left(\left|f_{\tau}(x, X_{\tau}^{N,i} )\right|^2+\left|f_{\tau}(x, X_{\tau}^{N,j} )\right|^2\right)\d x \right\|^{m}_{L^{m^{\prime}}([s,t])}\\
		\leq&\,C\E\left\|\int_{\R^d}\left|f_{\tau}(x, X_{\tau}^{N,i} )\right|^2\d x \right\|^{m}_{L^{m^{\prime}}([s,t])},
	\end{aligned}
\end{equation}
where the last inequality is because the processes $X_{\cdot}^{N,i}$ and $X_{\cdot}^{N,j}$ have the same law. Using the fact
$$1+|a+b|^{d+1}\leq C_d(1+|a|^{d+1})(1+|b|^{d+1})$$
and letting $y=x-X_{\tau}^{N,i}$, one has
\begin{equation*}
	\begin{aligned}
		\int_{\R^d}\left|f_{\tau}(x, X_{\tau}^{N,i} )\right|^2\d x&\,=\int_{\R^d}\big(1+|x|^{\frac{d+1}{2}} \big)^2 \big|\nabla e^{(t-\tau) A} V^{N}\big(x-X_{\tau}^{N,i}\big)\big|^2\d x\\
		&\,\leq C\int_{\R^d} \big(1+|y+X_{\tau}^{N,i}|^{d+1} \big) \big|\nabla e^{(t-\tau) A} V^{N} (y ) \big|^2\d y\\
		&\,\leq C\big(1+|X_{\tau}^{N,i}|^{d+1} \big)\int_{\R^d}\big(1+|y|^{d+1}\big) \big|\nabla e^{(t-\tau) A} V^{N} (y )\big|^2\d y.\\
	\end{aligned}
\end{equation*}
Therefore,
\begin{equation}\label{Efm}
	\begin{aligned}	
		&\,\E\left\|\int_{\R^d}\left|f_{\tau}(x, X_{\tau}^{N,i} )\right|^2\d x \right\|^{m}_{L^{m^{\prime}}([s,t])}
		\\
		\leq&\,C\E\left\|\big(1+|X_{\tau}^{N,i}|^{d+1} \big)\int_{\R^d}\big(1+|y|^{d+1}\big) \big|\nabla e^{(t-\tau) A} V^{N} (y )\big|^2\d y \right\|^{m}_{L^{m^{\prime}}([s,t])} \\
		\leq&\,C\E\bigg[\sup_{\tau \in[s, t]} \big(1+|X_{\tau}^{N,i}|^{m(d+1)} \big)\bigg] \left\|\int_{\R^d}\big(1+|y|^{d+1}\big) \big| \nabla e^{(t-\tau) A} V^{N} (y ) \big|^{2}\d y\right\|^{m}_{L^{m^{\prime}}([s,t])} \\
		\leq &\, C\left\|K_{\epsilon}\right\|_{\infty}^{m(d+1)} \left\|\int_{\R^d}\big(1+|y|^{d+1}\big) \big| \nabla e^{(t-\tau) A} V^{N} (y ) \big|^{2}\d y\right\|^{m}_{L^{m^{\prime}}([s,t])},
	\end{aligned}
\end{equation}
where in the last step we have used Lemma \ref{1stmoment} for any small enough $\eps>0$. Combining this estimate with  \eqref{EQL12} and \eqref{i1}, we obtain
$$\E\Big[\tilde{\xi}_N^{\, m/2} \Big] \leq \frac{C \|K_{\epsilon} \|_{\infty}^{d+ \frac{m(d+1)}2}} {N^{\frac14}} (t-s )^{\frac{1}{2}} \left\|\int_{\R^d}\big(1+|y|^{d+1}\big) \big| \nabla e^{(t-\tau) A} V^{N} (y ) \big|^{2}\d y\right\|^{\frac {m}2}_{L^{m^{\prime}}([s,t])} . $$
Recall the definition of $\tilde{\xi}_N$; inserting this estimate into \eqref{m1L11} gives us
\begin{equation*}%\label{Vny2}
	\begin{aligned}
		\tilde M_1 \leq&\, \frac{C\left\|K_{\epsilon}\right\|_{\infty}^{\frac{d}{m}+\frac{d+1}{2}}}{N^{\frac{1}{4m}}} (t-s )^{\frac{1}{2m}} \left\|\int_{\R^d}\big(1+|y|^{d+1}\big) \big| \nabla e^{(t-\tau) A} V^{N} (y )\big|^{2} \d y\right\|^{\frac{1}{2}}_{L^{m^{\prime}}([s,t])}\\
		\leq&\, \frac{C\left\|K_{\epsilon}\right\|_{\infty}^{d+1}}{N^{\frac{1}{4m}}} (t-s)^{\frac{1}{2m}} \left\|\int_{\R^d} \big| \nabla e^{(t-\tau) A} V^{N} (y ) \big|^{2}\d y\right\|^{\frac{1}{2}}_{L^{m^{\prime}}([s,t])} \\ &\,+\frac{C\left\|K_{\epsilon}\right\|_{\infty}^{d+1} }{N^{\frac{1}{4m}}} (t-s)^{\frac{1}{2m}} \left\|\int_{\R^d}|y|^{d+1} \big| \nabla e^{(t-\tau) A} V^{N} (y ) \big|^{2}\d y \right\|^{\frac{1}{2}}_{L^{m^{\prime}}([s,t])}.
	\end{aligned}
\end{equation*}
We denote the two quantities by $J_1$ and $J_2$. By \eqref{lpestimate},
\begin{equation*}%\label{lst}
	\begin{aligned}
		\big\| \nabla e^{(t-\tau) A} V^{N} \big\|_{L^{2}}
		=&\,\big\|(\mathrm{I}-A)^{-\frac{\delta}{2}} \nabla e^{(t-\tau) A}(\mathrm{I}-A)^{\frac{\delta}{2}} V^{N} \big\|_{L^{2}}
		\\\leq &\,  \frac{C_{\delta,\nu,T}}{(t-\tau)^{\frac{1-\delta}{2}}}\left\|V^{N}\right\|_{\delta, 2}
		\leq \frac{C_{\delta,\nu,T}}{(t-\tau)^{\frac{1-\delta}{2}}} N^{\beta\left(\frac{d}{2}+\delta\right)},
	\end{aligned}
\end{equation*}
the first term $J_1$ can be estimated as
\begin{equation}\label{M1J1}
	\begin{aligned}
		J_1\leq&\, C_{\delta}\left\|K_{\epsilon}\right\|_{\infty}^{d+1}\left(t-s\right)^{\frac{1}{2m}}N^{\beta\left(\frac{d}{2}+\delta\right)-\frac{1}{4m}}\left(\int_{s}^{t} \frac{1}{(t-\tau)^{m^{\prime}(1-\delta)}} \d \tau\right)^{\frac{1}{2m^{\prime}}}\\
		\leq&\, C_{\delta}\left\|K_{\epsilon}\right\|_{\infty}^{d+1} N^{\beta\left(\frac{d}{2}+\delta\right)-\frac{1}{4m}}\left(t-s\right)^{\frac{\delta}{2}}.
	\end{aligned}
\end{equation}
Now let us estimate the second term $J_2$. We will use the following result.

\begin{lemma}\label{y4}
	Assume Hypotheses \ref{hypothesis} and  $\frac{2}{m}<\delta<1$, then for any $N\in\N$,
	$$
	\int_{\R^d}|y|^{d+1}\big| \nabla e^{(t-\tau) A} V^{N}(y) \big|^{2}\d y\leq C_{\delta} N^{\beta(d+2\delta)}\frac{1}{(t-\tau)^{1-\delta}}.
	$$
\end{lemma}

We postpone the proof of Lemma \ref{y4} in Section \ref{sectionprooflemma}. Thus, the second term $J_2$ can be estimated as
\begin{equation}\label{M1J2}
	\begin{aligned}
		J_2\leq&\, C_{\delta}\left\|K_{\epsilon}\right\|_{\infty}^{d+1}\left(t-s\right)^{\frac{1}{2m}}N^{\beta\left(\frac{d}{2}+\delta\right)-\frac{1}{4m}}\left(\int_{s}^{t} \frac{1}{(t-\tau)^{m^{\prime}(1-\delta)}} \d \tau\right)^{\frac{1}{2m^{\prime}}}\\
		\leq&\, C_{\delta}\left\|K_{\epsilon}\right\|_{\infty}^{d+1} N^{\beta\left(\frac{d}{2}+\delta\right)-\frac{1}{4m}}\left(t-s\right)^{\frac{\delta}{2}}.
	\end{aligned}
\end{equation}
Therefore, adding \eqref{M1J1} and \eqref{M1J2}, we have
\begin{equation}\label{step11}
	\begin{aligned}
		\tilde M_1\leq C_{\delta}\left\|K_{\epsilon}\right\|_{\infty}^{d+1} N^{\beta\left(\frac{d}{2}+\delta\right)-\frac{1}{4m}}\left(t-s\right)^{\frac{\delta}{2}}.
	\end{aligned}
\end{equation}

\smallskip

\noindent\textbf{Step 2. The bound of $\tilde M_2$ in \eqref{m1m21}}. \smallskip

Similarly to what we have done in Step 1, we adopt the weight $w(x)=1+|x|^{\frac{d+1}{2}}$ so that we can apply BDG-type inequality in $L^2(\R^d)$:
{\small \begin{equation*}\label{m1L1}
		\begin{aligned}
			\tilde M_2^m\leq&\, \E\bigg\|\sum_{k} \int_{0}^{s}\frac{1}{N} \sum_{i=1}^{N}\sigma_{k}^{N}(X_{\tau}^{N,i})\cdot   \big(\nabla e^{(s-\tau) A} \big(e^{(t-s)A}-\mathrm{I}\big)V^{N}\big) \big(\cdot-X_{\tau}^{N,i}\big) \d W_{\tau}^{k} \bigg\|_{L^1}^m\\
			\leq&\,C_d \E\bigg\| \sum_{k} \int_{0}^{s}\frac{1}{N} \sum_{i=1}^{N}\sigma_{k}^{N}(X_{\tau}^{N,i})\cdot  w(\cdot) \big(\nabla e^{(s-\tau) A} \big(e^{(t-s)A}-\mathrm{I}\big)V^{N}\big) \big(\cdot-X_{\tau}^{N,i}\big) \d W_{\tau}^{k} \bigg\|_{L^2}^m\\
			\leq&\, C_d \E\Bigg[ \sum_{k}\int_{0}^{s} \bigg\|\frac{1}{N} \sum_{i=1}^{N}\sigma_{k}^{N}(X_{\tau}^{N,i})\cdot w(\cdot) \big(\nabla e^{(s-\tau) A} \big(e^{(t-s)A}-\mathrm{I}\big)V^{N}\big) \big(\cdot-X_{\tau}^{N,i}\big) \bigg\|_{L^{2}}^{2} \d \tau \Bigg]^{\frac{m}{2}}.\\
		\end{aligned}
\end{equation*} }

\noindent We follow the idea for estimating $\tilde M_1$, define
$$g_{\tau}(x, X_{\tau}^{N,i} ):= w(x) \big(\nabla e^{(s-\tau) A}\big(e^{(t-s)A}-\mathrm{I}\big) V^{N}\big)\left(x-X_{\tau}^{N,i}\right)$$
to replace $f_{\tau}(x, X_{\tau}^{N,i} )$ from \eqref{m1L11} to \eqref{Efm} and integrate from $0$ to $s$. Then, we can prove the estimate
\begin{equation}\label{M2}
	\begin{aligned}
		\tilde M_2
		\leq&\,\frac{C\left\|K_{\epsilon}\right\|_{\infty}^{d+1}}{N^{\frac{1}{4m}}}\left\|\int_{\R^d}\big(1+|y|^{d+1}\big) \big| \nabla e^{(s-\tau) A}\big(e^{(t-s)A}-\mathrm{I}\big) V^{N} (y )\big|^{2}\d y\right\|^{\frac{1}{2}}_{L^{m^{\prime}}([0,s])}.
	\end{aligned}
\end{equation}
To continue our proof, we present the following lemma and postpone its proof in Section \ref{sectionprooflemma}.

\begin{lemma}\label{1y4}
	Assume Hypotheses \ref{hypothesis} and $\frac{2}{m}<\delta<1$; we take $\kappa$ small enough such that $0<\kappa<1-\delta$, then for any $N\in\N$,
	$$
	\int_{\R^{d}} \!\big(1+|y|^{d+1}\big) \big|\nabla e^{(s-\tau) A}\big(e^{(t-s)A}-\mathrm{I}\big)V^{N}(y) \big|^{2} \d y\leq C_{\kappa,\delta} N^{\beta(d+\delta(2d+2-\frac{d-2}{1-\kappa}))} \frac{(t-s)^{\delta}}{(s-\tau)^{1-\frac{\delta}{1-\kappa}}}.
	$$
\end{lemma}

Since $\frac{\delta}{1-\kappa}>\frac{1}{m}$, the function $(0,s)\ni \tau \mapsto (s-\tau)^{m^{\prime}(\frac{\delta}{1-\kappa}-1)}$ is integrable; substituting the above estimate into \eqref{M2}, %and taking $\kappa$ small enough such that $\delta<1-\kappa$
we have
\begin{equation}\label{step21}
	\begin{aligned}
		\tilde M_2
		\leq&\,C_{\kappa,\delta}\left\|K_{\epsilon}\right\|_{\infty}^{d+1}N^{\frac{\beta}{2}(d+ \delta(2d+2-(d-2)))-\frac{1}{4m}}(t-s)^{\frac{\delta}{2}} \bigg(\int_{0}^{s} \frac{1}{(s-\tau)^{m^{\prime}(1-\frac{\delta}{1-\kappa})}} \d \tau\bigg)^{\frac{1}{2m^{\prime}}}\\
		\leq&\,  C_{\delta}\left\|K_{\epsilon}\right\|_{\infty}^{d+1}N^{\frac{\beta}{2}(d+ \delta(d+4))-\frac{1}{4m}}(t-s)^{\frac{\delta}{2}}.
	\end{aligned}
\end{equation}

\smallskip

\noindent\textbf{Step 3. The bound of \eqref{m1m21}}. \smallskip

\eqref{step11} and \eqref{step21} yield that
\begin{equation}
	\begin{aligned}
		\big\|Z^N_t-Z^N_s \big\|_{L^m(\Omega,L^1)}\leq&\, C_{\delta} \|K_{\epsilon} \|_{\infty}^{d+1}\Big(N^{\beta\left(\frac{d}{2}+\delta\right)-\frac{1}{4m}}+N^{\frac{\beta}{2}(d+ \delta(d+4))-\frac{1}{4m}}\Big) (t-s)^{\frac{\delta}{2}}\\
		\leq&\,  C_{\delta} \|K_{\epsilon} \|_{\infty}^{d+1}N^{\frac{\beta}{2}(d+ \delta(d+4))-\frac{1}{4m}}(t-s)^{\frac{\delta}{2}}.
	\end{aligned}
\end{equation}
Applying Lemma \ref{sup} and taking $\eta=\frac{\delta}{2}$ satisfying $m\eta>1$, we deduce that
\begin{equation*}%\label{z1}
	\bigg\| \sup_{t \in[0, T]}\left\| Z^N_t\right\|_{L^1\left(\R^{d}\right)} \bigg\|_{L^m(\Omega)}\leq C_{m,\delta,T}\left\|K_{\epsilon}\right\|_{\infty}^{d+1}N^{\frac{\beta}{2}(d+ \delta(d+4))-\frac{1}{4m}},
\end{equation*}
where $\frac{2}{m}<\delta<1$ and  $\beta$ satisfies
$$0<\beta\leq\frac{1}{4m(d+2)}<\frac{1}{2m(d+\delta(d+4))}.$$

\subsubsection{Proofs of Lemmas \ref{y4} and \ref{1y4}}\label{sectionprooflemma}

\begin{proof}[Proof of Lemma \ref{y4}]
	On the one hand, by the definition \eqref{semigroup} of semigroup $e^{(t-\tau) A}$, and $\int_{\R^d}g_{t-\tau}(y-z)\d z=1$,
	%From Appendix A.2 of \cite{ORT20b}, we know that
	\begin{equation*}
		\begin{aligned}
			\int_{\R^d}|y|^{d+1} \big| \nabla e^{(t-\tau) A} V^{N} (y )\big|^{2}\d y
			\leq&\,\int_{\R^d}|y|^{d+1}\bigg| \int_{\R^d}g_{t-\tau}(y-z)  \nabla V^{N} (z )\d z\bigg|^{2}\d y\\
			\leq&\,\int_{\R^d}|y|^{d+1} \int_{\R^d} \left|\nabla V^{N} (z )\right|^{2}g_{t-\tau}(y-z) \d z\d y.
		\end{aligned}
	\end{equation*}
	Notice that $V^N(z)=N^{d\beta}V(N^{\beta}z)$; changing the variable $z^{\prime}=N^{\beta}z$ and by Fubini's theorem,
	\begin{equation*}
		\begin{aligned}
			&\, \int_{\R^d}|y|^{d+1} \big| \nabla e^{(t-\tau) A} V^{N} (y )\big|^{2}\d y \\
			\leq&\,N^{\beta(d+2)}\int_{\R^d}|y|^{d+1} \int_{\R^d} |\nabla V (z^{\prime} ) |^{2}g_{t-\tau} \big(y-N^{-\beta} z^{\prime} \big) \d z^{\prime}\d y\\
			\leq&\,N^{\beta(d+2)}\int_{\R^d} |\nabla V (z^{\prime} ) |^{2} \d z^{\prime} \int_{\R^d} \big|y+ N^{-\beta} z^{\prime} \big|^{d+1} g_{t-\tau}(y)\d y.\\
		\end{aligned}
	\end{equation*}
	Since $V\in C_c^{\infty}(\R^d,\R_{+})$ by Hypotheses \ref{hypothesis}, we have
	\begin{equation}\label{y2Vn1}
		\begin{aligned}
			&\,\int_{\R^d}|y|^{d+1} \big| \nabla e^{(t-\tau) A} V^{N} (y )\big|^{2}\d y\\
			\leq&\,CN^{\beta(d+2)}\left(\int_{\R^d}  |\nabla V (z^{\prime} ) |^{2} \d z^{\prime}\int_{\R^d}|y|^{d+1} g_{t-\tau}(y) \d y+\int_{\R^d} \big| N^{-\beta} z^{\prime} \big|^{d+1} |\nabla V (z^{\prime} ) |^{2} \d z^{\prime}\right)\\
			\leq&\, CN^{\beta(d+2)}\left((t-\tau)^{\frac{d+1}{2}}+N^{-\beta(d+1)}\right)\\
			\leq&\, CN^{\beta(d+2)}.
		\end{aligned}
	\end{equation}
	On the other hand, as $\nabla g_{t-\tau}(z)=Cg_{t-\tau}(z)\frac{z}{t-\tau}$, one has
	\begin{equation*}
		\begin{aligned}
			\int_{\R^d}\! |y|^{d+1}\big| \nabla e^{(t-\tau) A} V^{N} (y )\big|^{2}\d y
			=&\,\int_{\R^d}|y|^{d+1}\left|\int_{\R^d}\nabla g_{t-\tau}(y-z) V^{N} (z ) \d z\right|^{2}\d y\\
			\leq&\,\frac{C}{(t-\tau)^2} \!\int_{\R^d}\! |y|^{d+1}\bigg|\! \int_{\R^d}\!\! |y-z | V^{N} (z )g_{t-\tau}(y-z) \d z\bigg|^{2}\d y.
		\end{aligned}
	\end{equation*}
	Cauchy-Schwarz inequality implies
	$$
	\begin{aligned}
		&\left| \int_{\R^d} |y-z | V^{N} (z ) g_{t-\tau}(y-z) \d z\right|^{2}\\
		\leq&\,\left(\int_{\R^d} |y-z |^2g_{t-\tau}(y-z) \d z\right)\left(\int_{\R^d} |V^{N} (z ) |^2g_{t-\tau}(y-z) \d z\right)\\
		\leq&\, C(t-\tau)\int_{\R^d} |V^{N} (z ) |^2g_{t-\tau}(y-z) \d z,
	\end{aligned}
	$$
	hence,
	\begin{equation*}
		\begin{aligned}
			\int_{\R^d}|y|^{d+1} \big| \nabla e^{(t-\tau) A} V^{N} (y )\big|^{2}\d y
			\leq&\,\frac{C}{t-\tau}\int_{\R^d}|y|^{d+1} \int_{\R^d} |V^{N} (z ) |^2g_{t-\tau}(y-z) \d z\d y
			\\\leq&\,\frac{CN^{d\beta}}{t-\tau}\int_{\R^d}|y|^{d+1} \int_{\R^d} | V (z^{\prime} ) |^{2} g_{t-\tau} \big(y- N^{-\beta} z^{\prime} \big) \d z^{\prime}\d y\\
			\leq&\,\frac{CN^{d\beta}}{t-\tau} \int_{\R^d} | V (z^{\prime} )|^{2}\d z^{\prime}
			\int_{\R^d} \big|y+ N^{-\beta}z^{\prime} \big|^{d+1} g_{t-\tau}(y) \d y,
		\end{aligned}
	\end{equation*}
	%Due to $$|y+\frac{z^{\prime}}{N^{\beta}}|^{d+1}\leq C\left(|y|^{d+1}+|\frac{z^{\prime}}{N^{\beta}}|^{d+1}\right),$$
	where in the last two steps we have changed variables ($z'= N^\beta z$) and used Fubini's theorem. Then,
	\begin{equation}\label{y2Vn2}
		\begin{aligned}
			&\,\int_{\R^d}|y|^{d+1} \big| \nabla e^{(t-\tau) A} V^{N} (y )\big|^{2}\d y
			%		\\\leq&\,\frac{CN^{d\beta}}{t-\tau}\int_{\R^d}\big(|y|^{d+1}+|\frac{z^{\prime}}{N^{\beta}}|^{d+1}\big) \int_{\R^d} \left| V\left(z^{\prime}\right)\right|^{2}g_{t-\tau}(y) \d z^{\prime}\d y
			\\\leq&\,\frac{CN^{d\beta}}{t-\tau}\left(\int_{\R^d} | V(z^{\prime} ) |^{2} \d z^{\prime}\int_{\R^d}|y|^{d+1} g_{t-\tau}(y) \d y + \int_{\R^d} \big|N^{-\beta} z^{\prime} \big|^{d+1} | V(z^{\prime} ) |^{2} \d z^{\prime}\right)
			\\\leq&\,\frac{CN^{d\beta}}{t-\tau}\big( (t-\tau )^{\frac{d+1}{2}} +N^{-(d+1)\beta}\big)\\
			\leq&\, CN^{d\beta}(t-\tau)^{-1}.
		\end{aligned}
	\end{equation}
	By interpolating \eqref{y2Vn1} and \eqref{y2Vn2} with powers $\delta$ and $1-\delta$, we complete the proof of Lemma \ref{y4}.
\end{proof}

\begin{proof}[Proof of Lemma \ref{1y4}]
	Firstly, %in \eqref{M2},
	$$
	\begin{aligned}
		\big|\nabla e^{(s-\tau) A}\big(e^{(t-s)A}-\mathrm{I}\big) V^{N} (y )\big|^{2} \leq\,2\big|\nabla e^{(t-\tau) A} V^{N}(y) \big|^{2} +2\big|\nabla e^{(s-\tau) A} V^{N}(y)\big|^{2}.
	\end{aligned}
	$$
	Hence,
	$$
	\begin{aligned}
		&\,\int_{\R^d}(1+|y|^{d+1}) \big|\nabla e^{(s-\tau) A}\big(e^{(t-s)A}-\mathrm{I}\big) V^{N} (y )\big|^{2} \d y \\
		\leq&\, 2\int_{\R^d}(1+|y|^{d+1}) \Big(\big|\nabla e^{(t-\tau) A} V^{N}(y)\big|^{2} +\big|\nabla e^{(s-\tau) A} V^{N}(y)\big|^{2}\Big) \d y \\
		=&\, 2\int_{\R^d}(1+|y|^{d+1}) \big|\nabla e^{(t-\tau) A} V^{N}(y)\big|^{2}\d y +2\int_{\R^d}(1+|y|^{d+1}) \big|\nabla e^{(s-\tau) A} V^{N}(y)\big|^{2} \d y \\
		=:&\,B_1+B_2.
	\end{aligned}
	$$
	Recalling \eqref{y2Vn2} in the proof of Lemma \ref{y4}, we already have
	\begin{equation*}
		\int_{\R^d}|y|^{d+1}\big| \nabla e^{(t-\tau) A} V^{N}\left(y\right)\big|^{2}\d y\leq CN^{d\beta}(t-\tau)^{-1}.
	\end{equation*}
	Observe that
	$$
	\int_{\R^d} \big|\nabla e^{(t-\tau) A} V^{N}(y) \big|^2\d y= \big\|\nabla e^{(t-\tau) A} V^{N} \big\|_{L^2}^2 \leq C(t-\tau)^{-1} \|V^N \|_{L^2}^2 \leq  CN^{d\beta}(t-\tau)^{-1} ,
	$$
	thus, it holds
	$$
	B_1 \leq C N^{d\beta}\left(t-\tau\right)^{-1} .
	$$
	Similarly, we have
	$$ B_2 \leq C N^{d\beta}\left(s-\tau\right)^{-1}.$$
	Since $\tau<s<t$, one has
	\begin{equation}\label{iv1}
		\int_{\R^{d}}(1+|y|^{d+1}) \big|\nabla e^{(s-\tau) A} \big(e^{(t-s)A}-\mathrm{I}\big)V^{N}(y)\big|^{2} \d y \leq C N^{d\beta}\left(s-\tau\right)^{-1}.
	\end{equation}
	
	Secondly, applying the gradient on $V^{N}$, one gets
	$$
	\begin{aligned}
		&\, \big|\nabla e^{(s-\tau) A}\big(e^{(t-s)A}-\mathrm{I}\big)V^{N}(y)\big| \\
		= &\,\bigg|\int_{\R^d} g_{s-\tau}(y-x) \bigg(\int_{\R^d} g_{t-s}(z)  \nabla V^{N}(x-z)\d z-\nabla V^{N}(x) \bigg)\d x\bigg| \\
		\leq &\, \int_{\R^d} g_{s-\tau}(y-x) \int_{\R^d} g_{t-s}(z)\left|\nabla V^{N}(x-z)-\nabla V^{N}(x)\right| \d z \d x \\
		\leq &\, \left\|\nabla^{2} V^{N}\right\|_{\infty} \int_{\R^d} g_{s-\tau}(y-x) \int_{\R^d} g_{t-s}(z)|z| \d z \d x.
	\end{aligned}
	$$
	Since $\left\|\nabla^{2} V^{N}\right\|_{\infty}\leq CN^{\beta(d+2)}$, we have
	$$
	\big|\nabla e^{(s-\tau) A}\big(e^{(t-s)A}-\mathrm{I}\big)V^{N}(y)\big|\leq C N^{\beta(d+2)} \sqrt{t-s}.
	$$
	Then, for any $\kappa \in (0,1)$,
	$$ \begin{aligned}
		&\int_{\R^{d}}\big(1+|y|^{d+1}\big)\big|\nabla e^{(s-\tau) A}\big(e^{(t-s)A}-\mathrm{I}\big)V^{N}(y) \big|^{2} \d y\\
		\leq &\, C \big(N^{\beta(d+2)} \sqrt{t-s}\big)^{2(1-\kappa)} \int_{\R^{d}}\big(1+|y|^{d+1}\big) \big|\nabla e^{(s-\tau) A} \big(e^{(t-s)A} -\mathrm{I}\big)V^{N}(y) \big|^{2\kappa} \d y \\
		\leq&\, C \big(N^{\beta(d+2)} \sqrt{t-s}\big)^{2(1-\kappa)}\left\{\int_{\R^d}\big(1+|y|^{d+1}\big) \big|e^{(t-\tau) A} \nabla V^N(y) \big|^{2\kappa} \d y\right.\\
		&\qquad\qquad\qquad\qquad\qquad\qquad\qquad\left. +\int_{\R^d}\big(1+|y|^{d+1}\big) \big|e^{(s-\tau) A} \nabla V^N(y)\big|^{2\kappa} \d y\right\}.
	\end{aligned}
	$$
	By \eqref{semigroup} and changing variables, we have
	$$
	\big|e^{(t-\tau)A}\nabla V^{N}(y)\big|\leq N^{\beta}\int_{\R^{d}}g_{t-\tau}\big(y- N^{-\beta}z^{\prime} \big) |\nabla V(z^{\prime}) |\d z^{\prime},
	$$
	therefore,
	\begin{equation*}
		\begin{aligned}
			&\,\int_{\R^{d}}\big(1+|y|^{d+1}\big) \big| e^{(t-\tau) A}\nabla V^{N}(y)\big|^{2\kappa} \d y\\
			\leq&\, N^{2\beta\kappa}\int_{\R^{d}} |\nabla V(z^{\prime}) |^{2\kappa}\d z^{\prime} \int_{\R^{d}}\big(1+ |y+ N^{-\beta} z^{\prime} \big|^{d+1}\big)g^{2\kappa}_{t-\tau}(y)\d y\\
			\leq&\, CN^{2\beta\kappa}\int_{\R^{d}} |\nabla V(z^{\prime}) |^{2\kappa}\d z^{\prime}\int_{\R^{d}} \big(1+ |y |^{d+1} +\big| N^{-\beta}z^{\prime} \big|^{d+1}\big)g^{2\kappa}_{t-\tau}(y)\d y\\
			\leq&\, C_{d,\kappa} N^{2\beta\kappa}.
		\end{aligned}
	\end{equation*}
	In the same way, we can deduce that
	\begin{equation*}
		\begin{aligned}
			\int_{\R^{d}}\big(1+|y|^{d+1}\big) \big|e^{(s-\tau) A}\nabla V^{N}(y) \big|^{2\kappa}\d y
			\leq\, C_{d,\kappa} N^{2\beta\kappa}.
		\end{aligned}
	\end{equation*}
	Summarizing the above calculations yields
	\begin{equation}\label{iv2}
		\int_{\R^{d}}\big(1+|y|^{d+1}\big) \big|\nabla e^{(s-\tau) A}\big(e^{(t-s)A}-\mathrm{I}\big)V^{N}(y) \big|^{2} \d y\leq C N^{2\beta(d+2-(d+1)\kappa)} (t-s)^{1-\kappa},
	\end{equation}
	where $C$ is a constant independent of $\tau,s,t$.
	
	Finally, we can interpolate between \eqref{iv1} and \eqref{iv2} with the powers $1-\frac{\delta}{1-\kappa}$ and $\frac{\delta}{1-\kappa}$:%, which satisfying $\delta<1-\kappa$ :
	$$
	\int_{\R^{d}}\!\! \big(1+|y|^{d+1}\big) \big|\nabla e^{(s-\tau) A}\! \big(e^{(t-s)A}-\mathrm{I}\big)V^{N}\! (y) \big|^{2} \d y\leq C_{\kappa,\delta} N^{\beta(d+\delta(2d+2-\frac{d-2}{1-\kappa}))} \frac{(t-s)^{\delta}}{(s-\tau)^{1-\frac{\delta}{1-\kappa}}}.
	$$
	Since $\frac{2}{m}<\delta<1$, we can take $\kappa$ small enough such that $\delta<1-\kappa$.	
\end{proof}

\section{Estimates on initial data}

\subsection{The boundedness of $\omega_0^N$}\label{boundomega0}

We now give a sketch of proof that $\{\omega^N_0 \}_{N\geq 1}$ is uniformly bounded in $L^m(\Omega, \L)$, in the special case $m\leq p$. We shall also assume that $p=3$; the same idea works for more general integer $p$, at the price of more complicated computations. Recall that
$$\omega^N_0(x)=V^N\ast S_0^N(x)= \frac1N \sum_{i=1}^{N}V^N(x-X_0^i) \geq 0.$$
Firstly, we have
\begin{align}\label{lm1}
	\big\|\omega^N_0  \big\|_{L^m(\Omega, L^1)} = \bigg\{\E\bigg[\int_{\R^d} \frac{1}{N}\sum_{i=1}^{N}V^N(x-X_0^i)\, \d x\bigg]^m \bigg\}^{\frac1m}=1.
\end{align}
Next, as $m\leq p$, it holds
\begin{equation}\label{jensen}
	\begin{aligned}
		\big\|\omega^N_0  \big\|_{L^m(\Omega, L^p)} \leq \big\|\omega^N_0  \big\|_{L^p(\Omega, L^p)} =\bigg[\E\int_{\R^d}\bigg(\frac{1}{N}\sum_{i=1}^{N}V^N(x-X_0^i)\bigg)^p\d x\bigg]^{\frac{1}{p}}.
	\end{aligned}
\end{equation}
From now on we take  $p=3$. One has
$$
\begin{aligned} &\,\E\bigg(\frac{1}{N}\sum_{i=1}^{N}V^N(x-X_0^i)\bigg)^3\\
	=&\, \frac{1}{N^3}\sum_{i=1}^{N}\E\big(V^N(x-X_0^i)^3\big)+\frac{3}{N^3}\sum_{i\neq j}\E\big(V^N(x-X_0^i) \big) \E\big(V^N(x-X_0^j)^2\big)\\
	&+\frac{1}{N^3}\sum_{i\neq j\neq k}\E\big(V^N(x-X_0^i) \big)\E\big(V^N(x-X_0^j)\big) \E\big(V^N(x-X_0^k) \big)\\
	\leq&\, \frac{1}{N^2}\E\big(V^N(x-X_0^1)^3\big) +\frac{3}{N}\E\big(V^N(x-X_0^1) \big) \E\big(V^N(x-X_0^2)^2\big) +\big[\E\big(V^N(x-X_0^1)\big)\big]^3.
\end{aligned}
$$
By Jensen's inequality,
$$\aligned
\E\big(V^N(x-X_0^1) \big) \E\big(V^N(x-X_0^2)^2\big) &\leq \big[\E\big(V^N(x-X_0^1)^3 \big)\big]^{\frac13} \big[\E\big(V^N(x-X_0^2)^3 \big)\big]^{\frac23} \\
&= \E\big(V^N(x-X_0^1)^3\big),
\endaligned $$
therefore,
$$\E\bigg(\frac{1}{N}\sum_{i=1}^{N}V^N(x-X_0^i)\bigg)^3 \leq \frac4N \E\big(V^N(x-X_0^1)^3\big) + \big[\E\big(V^N(x-X_0^1)\big)\big]^3.$$
We have
$$\aligned \E\big(V^N(x-X_0^1)^3\big) &=\int_{\R^d}N^{3\beta d}V^3(N^{\beta}(x-y))\omega_0(y)\,\d y =N^{2\beta d}(V^3)^N\ast \omega_{0}(x), \\
\E\big(V^N(x-X_0^1) \big) & =\int_{\R^d}N^{\beta d}V(N^{\beta}(x-y))\omega_0(y)\,\d y =V^N\ast \omega_{0}(x),
\endaligned $$
where $(V^3)^N$ is a convolution kernel defined similarly as $V^N$ in terms of $V^3$. Thus,
\begin{equation*}\label{e}
	\begin{aligned}
		\E\bigg(\frac{1}{N}\sum_{i=1}^{N}V^N(x-X_0^i)\bigg)^3
		\leq \frac{4}{N^{1-2\beta d}}(V^3)^N\ast \omega_{0}(x) +\big(V^N\ast \omega_{0}(x)\big)^3 .
	\end{aligned}
\end{equation*}
By our choice of $\beta$ it holds $1-2\beta d\geq 0$; substituting this estimate into \eqref{jensen} leads to ($p=3$)
$$\aligned
\big\|\omega^N_0  \big\|_{L^m(\Omega, L^3)} &\leq C \bigg[\int_{\R^d} \Big((V^3)^N\ast \omega_{0}(x) +\big(V^N\ast \omega_{0}(x)\big)^3 \Big)\, \d x\bigg]^{\frac{1}{3}} \\
&= C \Big[ \big\|(V^3)^N\ast \omega_{0} \big\|_{L^1} + \big\|V^N\ast \omega_{0} \big\|_{L^3}^3 \Big]^{\frac{1}{3}} \\
&\leq C\Big[ \big\|(V^3)^N \big\|_{L^1} \|\omega_0 \|_{L^1} + \big\|V^N \big\|_{L^1}^3 \big\|\omega_{0}\|_{L^3}^3 \Big]^{\frac{1}{3}}.
\endaligned $$
Since $\|\omega_0 \|_{L^1}= \big\|V^N \big\|_{L^1}=1$, we finally get
$$\big\|\omega^N_0  \big\|_{L^m(\Omega, L^3)} \leq C\big[ \|V \|_{L^3}^3 + \|\omega_{0}\|_{L^3}^3 \big]^{\frac{1}{3}} \leq C\big[ \|V \|_{L^3} + \|\omega_{0}\|_{L^3} \big]. $$
Combining with \eqref{lm1}, we get the boundedness of $\{\omega^N_0 \}_{N\geq 1}$ in $L^m(\Omega, L^1\cap L^3)$.

\subsection{The distance between $\omega_0$ and $\omega_0^N$}\label{boundomega0omega}

For simplicity, we take $m=p=4$ and assume $\omega_{0}\in C_c^1 (\R^d)$; then
\begin{align*}
	\E\big\|\omega_{0}^{N}-\omega_{0}\big\|_{L^4}^4 &= \E\int_{\R^d}\bigg(\frac{1}{N}\sum_{i=1}^{N}V^N(x-X_0^i) -\omega_0(x)\bigg)^4 \d x \\
	&=\frac{1}{N^4}\E\int_{\R^d} \bigg(\sum_{i=1}^{N}\big(V^N(x-X_0^i)-\omega_0(x)\big)\bigg)^4\d x.
\end{align*}
Introducing the notation $Y_i(x)= V^N(x-X_0^i)-\omega_0(x),\, x\in \R^d$, which are i.i.d. random variables for $ 1\leq i\leq N$; we have
$$\aligned
\E\big\|\omega_{0}^{N}-\omega_{0}\big\|_{L^4}^4 & = \frac{1}{N^4} \E \int_{\R^d} \bigg[\sum_{i=1}^N Y_i(x)^4 + 4\sum_{i\neq j} Y_i(x)^3 Y_j(x)+ 3\sum_{i\neq j} Y_i(x)^2 Y_j(x)^2 \\
&\hskip35pt + 6\sum_{i\neq j\neq k}\! Y_i(x)^2 Y_j(x) Y_k(x) +\! \sum_{i\neq j\neq k\neq l}\! Y_i(x) Y_j(x) Y_k(x) Y_l(x) \bigg] \d x .
\endaligned $$
We denote the five terms by $I_1,\ldots, I_5$ respectively.

First, we have
$$\aligned
I_1&= \frac{1}{N^3} \int_{\R^d} \E Y_1(x)^4\,\d x \\
&= \frac{1}{N^3} \int_{\R^d} \E \Big[V^N(x-X_0^1)^4 - 4\omega_0(x)V^N(x-X_0^1)^3 + 6 \omega_0(x)^2 V^N(x-X_0^1)^2 \\
&\hskip60pt -4 \omega_0(x)^3 V^N(x-X_0^1) + \omega_0(x)^4 \Big]\,\d x ,
\endaligned$$
which is dominated by
$$\frac{1}{N^3} \int_{\R^d} \E \Big[V^N(x-X_0^1)^4 + 6 \omega_0(x)^2 V^N(x-X_0^1)^2 + \omega_0(x)^4 \Big]\,\d x=: I_{1,1} + I_{1,2} +I_{1,3}. $$
It is clear that $I_{1,3}= \frac1{N^3} \|\omega_0 \|_{L^4}^4$. One has
$$\aligned
I_{1,1} &= \frac{1}{N^3} \int_{\R^d} \int_{\R^d} V^N(x-y)^4 \omega_0(y)\,\d y\d x \\
&= \frac1{N^{3-3\beta d}} \int_{\R^d} \omega_0(y)\,\d y \int_{\R^d} N^{\beta d} V^4(N^\beta(x-y))\,\d x\\
&= \frac1{N^{3-3\beta d}} \|V\|_{L^4}^4.
\endaligned $$
In the same way,
$$\aligned
I_{1,2} &= \frac{6}{N^3} \int_{\R^d} \omega_0(x)^2 \int_{\R^d} V^N(x-y)^2 \omega_0(y)\,\d y\d x \\
&\leq \frac6{N^{3-\beta d}} \|\omega_0\|_\infty \int_{\R^d} \omega_0(x)^2\,\d x \int_{\R^d} N^{\beta d} V^2(N^\beta(x-y))\,\d y\\
&= \frac6{N^{3-\beta d}} \|\omega_0\|_\infty \|\omega_0 \|_{L^2}^2 \|V\|_{L^2}^2.
\endaligned $$
To sum up, we obtain
$$I_1 \leq \frac{C}{N^{3-3\beta d}} \big(\|\omega_0 \|_{L^4}^4 + \|V\|_{L^4}^4 + \|\omega_0\|_\infty \|\omega_0 \|_{L^2}^2 \|V\|_{L^2}^2 \big)= \frac{C_{V,\omega_0}}{N^{3-3\beta d}}. $$

Next, by independence of $\{Y_i(x)\}_{i=1}^N$ and Jensen's inequality,
$$\aligned
I_2&= \frac{4}{N^4} N(N-1) \int_{\R^d} \E Y_1(x)^3\, \E Y_2(x)\,\d x\\
&\leq \frac 4{N^2} \int_{\R^d} \big[\E Y_1(x)^4\big]^{\frac34} \big[\E Y_2(x)^4\big]^{\frac14}\,\d x \\
&= \frac 4{N^2} \int_{\R^d} \E Y_1(x)^4 \,\d x \leq  \frac{C_{V,\omega_0}}{N^{2-3\beta d}} .
\endaligned $$
Similarly,
$$I_3\leq \frac{C_{V,\omega_0}}{N^{2-3\beta d}}, \quad I_4 \leq \frac{C_{V,\omega_0}}{N^{1-3\beta d}}. $$

Finally, it remains to estimate $I_5$. We have
\begin{align*}
	I_5 &=\frac{1}{N^4}\int_{\R^d} \sum_{i\neq j\neq k\neq l } \E Y_i(x)\, \E Y_j(x)\,\E Y_k(x)\,\E Y_l(x)\, \d x \\
	&\leq  \int_{\R^d} \big[\E Y_1(x) \big]^4\, \d x = \int_{\R^d} \big[ V^N\ast \omega_0(x)-\omega_0(x) \big]^4\, \d x.
\end{align*}
Recall that $\omega_0$ and $V$ are compactly supported; we can find big $R>0$ such that $\omega_0$ and $V^N\ast \omega_0$ vanish outside $B(R)$ for all $N\geq 1$. Thus,
$$I_5 \leq \int_{B(R)} \big[ V^N\ast \omega_0(x)-\omega_0(x) \big]^4\, \d x.$$
Moreover,
\begin{equation*}\label{L4estimate}
	\begin{aligned}
		\big|V^N\ast \omega_0(x)-\omega_0(x)\big| & = \bigg|\int_{\R^d}V^N(x-y)(\omega_{0}(y)-\omega_{0}(x))\, \d y \bigg|\\
		&\leq \|\nabla\omega_0\|_\infty \int_{\R^d}N^{\beta d}V(N^{\beta}(x-y))|x-y|\, \d y\\
		& \leq\frac{C_{\omega_{0},V}}{N^{\beta}}\int_{\R^d}N^{\beta d}V(N^{\beta}(x-y))\, \d y = \frac{C_{\omega_{0},V}}{N^{\beta}} ,
	\end{aligned}
\end{equation*}
where the third step is due to that $V$ is compactly supported. As a result, $I_5\leq C/N^{4\beta}$ for some constant $C$ depending also on $R$. Combining the estimates of $I_1,\ldots,I_5$ above, we can find $\lambda>0$ such that
\begin{align*}
	\big\|\omega_{0}^{N}-\omega_{0}\big\|_{L^4(\Omega,L^4)}\leq \frac{C}{N^{\lambda}}.
\end{align*}

Regarding the $L^4(\Omega,L^1)$ norm of $\omega_{0}^{N}-\omega_{0}$, we note that the latter is also compactly supported, thus there is some $R>0$ such that
\begin{align*}
	\big\|\omega_{0}^{N}-\omega_{0}\big\|_{L^4(\Omega,L^1)}^4 &=\E \bigg[\int_{B(R)} \big|V^N\ast S_0^N(x)-\omega_0(x)\big|\,\d x \bigg]^4 \\
	&\leq C_R \int_{B(R)} \E \big|V^N\ast S_0^N(x)-\omega_0(x)\big|^4 \, \d x \\
	&=\frac{C_R}{N^4} \int_{B(R)} \E\bigg(\sum_{i=1}^{N} Y_i(x) \bigg)^4\, \d x,
\end{align*}
so we can repeat the above computations to get the same estimate. In conclusion, we obtain the rate of convergence:
\begin{equation*}
	\zeta_N:=\big\|\omega^N_0 - \omega_0 \big\|_{L^4(\Omega, L^1\cap L^4)}\leq \frac{C}{N^{\lambda}} \rightarrow 0  \quad\text{as } N \rightarrow\infty.
\end{equation*}

\end{appendix}

\bigskip

\noindent \textbf{Acknowledgements.} The second named author is grateful to the financial supports of the National Key R\&D Program of China (No. 2020YFA0712700), the National Natural Science Foundation of China (Nos. 11931004, 12090014), and the Youth Innovation Promotion Association, CAS (Y2021002).

\end{document}